\documentclass[12pt]{amsart}
\usepackage{amsfonts,amssymb,amscd,amsmath,enumerate,verbatim}
\usepackage[latin1]{inputenc}
\usepackage{amscd}
\usepackage{latexsym}
\usepackage{mathptmx}
\usepackage{multicol}
\input xy
\xyoption{all}

\def\frk{\mathfrak}               

\def\Phi{{\frk N}}
\def\opn#1#2{\def#1{\operatorname{#2}}} 
\opn\chara{char} 
\opn\length{\ell} 
\opn\pd{pd} 
\opn\rk{rk}
\opn\projdim{proj\,dim} 
\opn\injdim{inj\,dim} 
\opn\rank{rank}
\opn\depth{depth} 
\opn\grade{grade} 
\opn\height{height}
\opn\embdim{emb\,dim} 
\opn\codim{codim}

\opn\Tr{Tr} 
\opn\bigrank{big\,rank}
\opn\superheight{superheight}
\opn\lcm{lcm}
\opn\trdeg{tr\,deg}
\opn\reg{reg} 
\opn\lreg{lreg} 
\opn\ini{in} 
\opn\lpd{lpd}
\opn\size{size}
\opn\mult{mult}
\opn\dist{dist}
\opn\cone{cone}
\opn\lex{lex}
\opn\rev{rev}
\opn\im{im}
\opn\m{m}
\opn\div{div} \opn\Div{Div} \opn\cl{cl} \opn\Cl{Cl}
\opn\Spec{Spec} \opn\Supp{Supp} \opn\supp{supp} \opn\Sing{Sing}
\opn\Ass{Ass} \opn\Min{Min}
\opn\Ann{Ann} \opn\Rad{Rad} \opn\Soc{Soc}
\opn\Syz{Syz} \opn\Im{Im} \opn\Ker{Ker} \opn\Coker{Coker}
\opn\Am{Am} \opn\Hom{Hom} \opn\Tor{Tor} \opn\Ext{Ext}
\opn\End{End} \opn\Aut{Aut} \opn\id{id} \opn\ini{in}

\opn\nat{nat}
\opn\pff{pf}
\opn\Pf{Pf} \opn\GL{GL} \opn\SL{SL} \opn\mod{mod} \opn\ord{ord}
\opn\Gin{Gin}
\opn\Hilb{Hilb}\opn\adeg{adeg}\opn\std{std}\opn\ip{infpt}
\opn\Pol{Pol}
\opn\sat{sat}
\opn\Var{Var}
\opn\Gen{Gen}

\opn\aff{aff} \opn\con{conv} \opn\relint{relint} \opn\st{st}
\opn\lk{lk} \opn\cn{cn} \opn\core{core} \opn\vol{vol}
\opn\link{link} \opn\star{star}
\opn\gr{gr}

\def\pot#1#2{#1[\kern-0.28ex[#2]\kern-0.28ex]}

\opn\dirlim{\underrightarrow{\lim}}
\opn\inivlim{\underleftarrow{\lim}}
\def\Implies{\ifmmode\Longrightarrow \else
        \unskip${}\Longrightarrow{}$\ignorespaces\fi}
\def\implies{\ifmmode\Rightarrow \else
        \unskip${}\Rightarrow{}$\ignorespaces\fi}
\def\iff{\ifmmode\Longleftrightarrow \else
        \unskip${}\Longleftrightarrow{}$\ignorespaces\fi}

\let\:=\colon
\newtheorem{Theorem}{Theorem}[section]
\newtheorem{Lemma}[Theorem]{Lemma}
\newtheorem{Corollary}[Theorem]{Corollary}
\newtheorem{Proposition}[Theorem]{Proposition}

\newtheorem{Question}[Theorem]{Question}

\let\epsilon\varepsilon
\let\phi=\varphi
\let\kappa=\varkappa
\def\qed{\ifhmode\textqed\fi
      \ifmmode\ifinner\quad\qedsymbol\else\dispqed\fi\fi}
\def\textqed{\unskip\nobreak\penalty50
       \hskip2em\hbox{}\nobreak\hfil\qedsymbol
       \parfillskip=0pt \finalhyphendemerits=0}
\def\dispqed{\rlap{\qquad\qedsymbol}}

\opn\dis{dis}
\opn\height{height}
\opn\dist{dist}
\def\pnt{{\raise0.5mm\hbox{\large\bf.}}}

\opn\Lex{Lex}

\begin{document}

\title{On the three graph invariants related to matching of finite simple graphs}
\author{Kazunori Matsuda and Yuichi Yoshida}

\address{Kazunori Matsuda,
Kitami Institute of Technology, 
Kitami, Hokkaido 090-8507, Japan}
\email{kaz-matsuda@mail.kitami-it.ac.jp}

\address{Yuichi Yoshida, 
Kitami Institute of Technology, 
Kitami, Hokkaido 090-8507, Japan}
\email{f1812101792@std.kitami-it.ac.jp}

\begin{abstract}
Let $G$ be a finite simple graph on the vertex set $V(G)$ and let $\text{ind-match}(G)$, $\text{min-match}(G)$ and $\text{match}(G)$ denote the induced matching number, the minimum matching number and the matching number of $G$, respectively. 
It is known that the inequalities $\text{ind-match}(G) \leq \text{min-match}(G) \leq \text{match}(G) \leq 2\text{min-match}(G)$ and $\text{match}(G) \leq \left\lfloor |V(G)|/2  \right\rfloor$ hold in general. 

In the present paper, we determine the possible tuples 
$(p, q, r, n)$ with $\text{ind-match}(G) = p$, $\text{min-match}(G) = q$, $\text{match}(G) = r$ and $|V(G)| = n$ arising from connected simple graphs.  
As an application of this result, we also determine the possible tuples 
$(p', q, r, n)$ with $\reg(G) = p'$, $\text{min-match}(G) = q$, $\text{match}(G) = r$ and $|V(G)| = n$ arising from connected simple graphs, where $I(G)$ is the edge ideal of $G$ and $\reg(G) = \reg(K[V(G)]/I(G))$ is the Castelnuovo--Mumford regularity of the quotient ring $K[V(G)]/I(G)$. 
\end{abstract}

\maketitle

\section*{Introduction}
Let $G = (V(G), E(G))$ be a finite simple graph on the vertex set $V(G)$ with the edge set $E(G)$. 
The main topic of this paper is graph-theoretical invariants related to {\em matching}. 
\begin{itemize} 
	\item A subset $M = \{e_{1}, \ldots, e_{s}\} \subset E(G)$ is said to be a {\em matching} of $G$ if $e_{i} \cap e_{j} = \emptyset$ for all $1 \leq i < j \leq s$. For a matching $M$, we write $V(M) = \{v : v \in e \ \text{for some} \ e \in M\}$. A {\em perfect matching} $M$ is a matching of $G$ with $V(M) = V(G)$. 
	\item A matching of $M$ of $G$ is {\em maximal} if $M \cup \{e\}$ cannot be a matching of $G$ for all $e \in E(G) \setminus M$. 
	Note that a matching $M$ is maximal if and only if $V(G) \setminus V(M)$ is an independent set of $G$. 
	\item A matching $M = \{e_{1}, \ldots, e_{s}\} \subset E(G)$ of $G$ is said to be an {\em induced matching} if, for all $1 \leq i < j \leq s$, there is no edge $e \in E(G)$ with $e \cap e_{i} \neq \emptyset$ and $e \cap e_{j} \neq \emptyset$. 
	\item The {\em matching number} \text{match}(G), the {\em minimum matching number} \text{min-match}(G) and the {\em induced matching number} \text{ind-match}(G) of $G$ are defined as follows respectively:
	\begin{eqnarray*}
	\text{match}(G)&=&\max\{|M| : M \text{\ is\ a\ matching\ of\ } G \}; \\
	\text{min-match}(G)&=&\min\{|M| : M \text{\ is\ a\ maximal\ matching\ of\ } G \}; \\
	\text{ind-match}(G)&=&\max\{|M| : M \text{\ is\ an\ induced\ matching\ of\ } G \}.  
	\end{eqnarray*}
\end{itemize}

Hall's marriage theorem \cite{Hall} says that, for any bipartite graph $G$ on the bipartition 
$V(G) = X \cup Y$ with $|X| \leq |Y|$, $\text{match}(G) = |X|$ holds if and only if 
$|N_{G}(S)| \geq |S|$ for all $S \subset X$, where $N_{G}(S) = \bigcup_{v \in S} N_{G}(v)$ 
and $N_{G}(v) = \{ w \in V(G) : \{v, w\} \in E(G) \}$.  
Besides this, there are many previous studies for $\text{match}(G)$, $\text{min-match}(G)$ 
and $\text{ind-match}(G)$, see \cite{AA, AV, BCL, DK, FR, Romeo, SMWZ}. 
To explain out motivation, we focus on the two known results as below.  

First, in \cite{HHKT}, it is proven that the inequalities 
\begin{equation}\label{3match}
\text{ind-match}(G) \leq \text{min-match}(G) \leq \text{match}(G) \leq 2\text{min-match}(G)
\end{equation}
hold for all simple graph $G$ and a classification of connected simple graphs $G$ satisfying $\text{ind-match}(G) =  \text{min-match}(G) = \text{match}(G)$ is given 
(\cite[Theorem 1]{CW}, \cite[Remark 0.1]{HHKO}) and such graphs are studied in \cite{HHKO, HKKMVT, HKMT, HKMVT, SF, T}. 
A classification of connected simple graphs $G$ with $\text{ind-match}(G) =  \text{min-match}(G)$ is also given in \cite{HHKT}. 

Second, by definition of matching, the equality 
\begin{equation}\label{UB_match}
\text{match}(G) \leq \left\lfloor |V(G)|/2 \right\rfloor
\end{equation}
holds and some classes of graphs $G$ with $\text{match}(G) = \left\lfloor |V(G)|/2 \right\rfloor$ are given (see \cite{CGH, GR, LV, Sumner}).

From (\ref{3match}) and (\ref{UB_match}), it is natural to arise the following question:  

\begin{Question}\label{Question}
Is there any connected simple graph $G = G(p, q, r, n)$ such that
\[ 
{\rm{ind}}\mathchar`-{\rm{match}}(G) = p, {\rm{min}}\mathchar`-{\rm{match}} = q, {\rm{match}}(G) = r \ and \ |V(G)| = n
\]
for all integers $p, q, r, n$ with $1 \leq p \leq q \leq r \leq 2q$ and $ r \leq \left\lfloor n/2 \right\rfloor$ ? 
\end{Question}

As a previous study related to Question \ref{Question}, a connected simple graph $G = G(p, q, r)$ such that $\text{ind-match}(G) = p, \text{min-match}(G) = q$ and  $\text{match}(G) = r$ was constructed for all integers $p, q, r$ with $1 \leq p \leq q \leq r \leq 2q$ (see \cite[Theorem 2.3]{HHKT}). 
Since the graph $G$ constructed in the proof of \cite[Theorem 2.3]{HHKT} satisfies $|V(G)| = 2\text{match}(G) + \text{ind-match}(G) - 1$, for example, we can see that there exists a connected simple graph $G = G(2, 3, 4, 9)$ such that $\text{ind-match}(G) = 2, \text{min-match}(G) = 3, \text{match}(G) = 4$ and $|V(G)| = 9$ by virtue of \cite[Theorem 2.3]{HHKT}. 
However, there is a connected simple graph $G = G(2, 3, 4, 8)$ such that $\text{ind-match}(G) = 2, \text{min-match}(G) = 3, \text{match}(G) = 4$ and $|V(G)| = 8$; see Figure \ref{fig:2348}. 
This graph says that it is necessary to construct new family of connected simple graphs  
in order to solve Question \ref{Question}.  

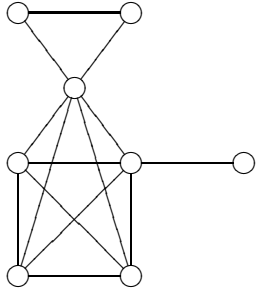
\begin{figure}[htbp]
\centering
\bigskip

	\begin{xy}
		\ar@{} (0,0);(70, -5) *\cir<4pt>{} = "A"
		\ar@{-} "A";(70, -20) *\cir<4pt>{} = "E";
		\ar@{-} "A";(85, -5) *\cir<4pt>{} = "B";
		\ar@{-} "A";(85, -20) *\cir<4pt>{} = "D";
		\ar@{-} "B";"D";
		\ar@{-} "B";"E";
		\ar@{-} "D";"E";
		\ar@{} (0,0);(70, 15) *\cir<4pt>{} = "C";
		\ar@{-} "C";(85, 15) *\cir<4pt>{} = "F";
		\ar@{} (0,0);(77.5, 5) *\cir<4pt>{} = "G";
		\ar@{-} "A";"G";
		\ar@{-} "B";"G";
		\ar@{-} "C";"G";
		\ar@{-} "D";"G";
		\ar@{-} "E";"G";
		\ar@{-} "F";"G";
		\ar@{-} "B";(100, -5) *\cir<4pt>{} = "H";
	\end{xy}

  \caption{A connected graph $G = G(2, 3, 4, 8)$ with $\text{ind-match}(G) = 2$, $\text{min-match}(G) = 3$, $\text{match}(G) = 4$ and $|V(G)| = 8$. }
  \label{fig:2348}
\end{figure}

Based on the above, we state main results of the present paper. 
As the first result is, we determine the possible tuples
\[
(\text{ind-match}(G), \text{min-match}(G), \text{match}(G), |V(G)|)
\] 
arising from connected simple graphs. 

\begin{Theorem}\label{Main1}(see Theorem \ref{first-main})\ 
Let $n \geq 2$ be an integer and set 
\begin{eqnarray*}
& & {\rm{\bf Graph}_{ind\text{-}match, min\text{-}match, match}}(n) \\
&=& \left\{(p, q, r) \in \mathbb{N}^{3} ~\left|~
\begin{array}{c}
  \mbox{\rm{There\ exists\ a\ connected\ simple\ graph}\ $G$ \ \rm{with}\ $|V(G)| = n$} \\
  \mbox{{\rm{and}} \ $\text{\rm ind-match}(G) = p, \ \text{\rm min-match}(G) = q, \ \text{\rm match}(G) = r$} \\ 
\end{array}
\right \}\right. .  
\end{eqnarray*}
Then we have the following: 
\begin{enumerate}
	\item[$(1)$] If $n$ is odd, then 
	\begin{eqnarray*}
	& & {\rm{\bf Graph}_{ind\text{-}match, min\text{-}match, match}}(n) \\
	&=& \left\{ (p, q, r) \in \mathbb{N}^{3} \ \middle| \ 1 \leq p \leq q \leq r \leq 2q \ \ {\rm{and}} \ \ r \leq \frac{n-1}{2} \right\} . 
	\end{eqnarray*}
	\item[$(2)$] If $n$ is even, then 
	\begin{eqnarray*}
	& & {\rm{\bf Graph}_{ind\text{-}match, min\text{-}match, match}}(n) \\ 
	&=& \left\{ (1, q, r) \in \mathbb{N}^{3} \ \middle| \ 1 \leq q \leq r \leq 2q \ \ {\rm{and}} \ \ r \leq \frac{n}{2} \right\} \\
	&\cup& \left\{ (p, q, r) \in \mathbb{N}^{3} \ \middle| \ 2 \leq p \leq q \leq r \leq 2q, \ r \leq \frac{n}{2}\ \ {\rm{and}}\ \ (q, r) \neq \left( \frac{n}{2}, \frac{n}{2} \right) \right\} . 
	\end{eqnarray*}
\end{enumerate}
\end{Theorem}
Note that this theorem gives a complete answer for Question \ref{Question}. 


The second main result is related to the invariants of edge ideals. 
Let $G$ be a finite simple graph on the vertex set $V(G) = \left\{ x_{1}, \ldots, x_{|V(G)|} \right\}$ and $E(G)$ the set of edges of $G$. 
Let $K[V(G)] = K\left[ x_{1}, \ldots, x_{|V(G)|} \right]$ be the polynomial ring in $|V(G)|$ variables over a field $K$. 
The {\em edge ideal} of $G$, denoted by $I(G)$, is the ideal of $K[V(G)]$ generated by quadratic monomials $x_{i}x_{j}$ associated with $\{x_{i}, x_{j}\} \in E(G)$.   
Among the current trends in combinatorial commutative algebra, the edge ideal is one of the main topic and has been studied multilaterally by many researchers, see \cite{CRT, CRTY, DHS, HaVanTuyl, KM, MV, NeP, SVV, Vi2001, W}. 

Let $\reg(G) = \reg(K[V(G)]/I(G))$ denote the {\em Castelnuovo--Mumford regularity} (regularity for short, see \cite[Section 18]{P}) of the quotient ring $K[V(G)]/I(G)$. 
It is known that 
\[
\text{ind-match}(G) \leq \reg(G) \leq \text{min-match}(G)
\]
holds for all simple graph $G$ (the lower bound was given by Katzman \cite{K} and the upper bound was given by Woodroofe \cite{W}).  
Moreover, H\`{a}--Van Tuyl proved that $\text{ind-match}(G) = \reg(G)$ holds if $G$ is a chordal graph (\cite[Corollary 6.9]{HaVanTuyl}). 
By virtue of this result together with Theorem \ref{Main1}, we also determine the possible tuples 
\[
(\reg(G), \text{min-match}(G), \text{match}(G), |V(G)|)
\]
arising from connected simple graphs. The second main result is as follows. 

\begin{Theorem}\label{Main2} (see Theorem \ref{second-main})\ 
Let $n \geq 2$ be an integer and set 
\begin{eqnarray*}
& & {\rm{\bf Graph}_{reg, min\text{-}match, match}}(n) \\
&=& \left\{(p', q, r) \in \mathbb{N}^{3} ~\left|~
\begin{array}{c}
  \mbox{\rm{There\ exists\ a\ connected\ simple\ graph}\ $G$\ {\rm{with}}\ $|V(G)| = n$ } \\
  \mbox{{\rm{and}} \ $\text{\rm reg}(G) = p', \ \text{\rm min-match}(G) = q, \ \text{\rm match}(G) = r$} \\ 
\end{array}
\right \}\right. .  
\end{eqnarray*}
Then one has 
\[
{\rm{\bf Graph}_{reg, min\text{-}match, match}}(n) = {\rm{\bf Graph}_{ind\text{-}match, min\text{-}match, match}}(n). 
\]
\end{Theorem}

Our paper is organized as follows. 
In Section 1, we prepare some lemmas and propositions in order to prove main results. 
In Section 2, we give a proof of Theorem \ref{Main1}. 
In Section 3, we introduce previous studies related to Theorem \ref{Main2} and give a proof.  


\section{Preparation}

In this section, we prepare some lemmas and propositions in order to prove our main results. 

\subsection{Known results}
In this subsection, we present known results related to our study.  
First, we recall two important inequalities. 

\begin{Proposition}\label{important}
Let $G$ be a finite simple graph on the vertex set $V(G)$. Then 
\begin{enumerate}
	\item[$(1)$] ${\rm{ind}}\mathchar`-{\rm{match}}(G) \leq {\rm{min}}\mathchar`-{\rm{match}}(G) \leq {\rm{match}}(G) \leq 2{\rm{min}}\mathchar`-{\rm{match}}(G)$. 
	\item[$(2)$] ${\rm{match}}(G) \leq \left\lfloor |V(G)|/2 \right\rfloor$. 
\end{enumerate}
\end{Proposition}
\begin{proof}
(1) : See \cite[Proposition 2.1 and Remark 3.2]{HHKT}. \\
\ \ (2) : Let $M$ be a matching of $G$ with $|M| = \text{match}(G)$. 
Then $|V(G)| \geq |V(M)| = 2\text{match}(G)$ by the definition of matching.  
Hence we have $\text{match}(G) \leq \left\lfloor |V(G)|/2 \right\rfloor$. 
\end{proof}

Next, we recall the definition of the $S$-suspension introduced in \cite{HKM}. 
A subset $S \subset V(G)$ is said to be an {\em independent set} of $G$ if $\{u, v\} \not\in E(G)$ for all $u, v \in S$. 
Note that the empty set $\emptyset$ is an independent set of $G$.   
For an independent set $S$ of $G$, we define the graph $G^{S}$ as follows: 
\begin{itemize}
	\item $V(G^{S}) = V(G) \cup \{w\}$, where $w$ is a new vertex. 
	\item $E(G^{S}) = \left\{ \{v, w\} : v \not\in S \right\}$. 
\end{itemize}
We call $G^{S}$ the {\em S-suspension} of $G$. 

\begin{Lemma}[{\cite[Lemma 1.5]{HKM}}]
\label{S-suspension}
Let $G$ be a finite simple graph on the vertex set $V(G)$. Suppose that $G$ has no isolated vertices. Let $S \subset V(G)$ be an independent set of $G$. Then ${\rm ind}\text{-}{\rm match}(G^{S}) = {\rm ind}\text{-}{\rm match}(G)$ holds. 
\end{Lemma}

As the end of this subsection, we recall a classification of connected simple graphs $G$ with $\text{min-match}(G) = |V(G)|/2$ given by Arumugam--Velammal. 

\begin{Proposition}[{\cite[Theorem 2.1]{AV}}]
\label{min = n/2}
Let $n \geq 2$ be an even integer and let $G = (V(G), E(G))$ be a connected simple graph with $|V(G)| = n$. 
Assume that ${\rm{min}}\mathchar`-{\rm{match}}(G) = n/2$. 
Then $G$ is either $K_{n}$ or $K_{n/2, n/2}$. 
In particular, ${\rm{ind}}\mathchar`-{\rm{match}}(G) = 1$ and there is no connected simple graph $G$ with $\left( {\rm{ind}}\mathchar`-{\rm{match}}(G), {\rm{min}}\mathchar`-{\rm{match}}(G), {\rm match}(G) \right) = (p, n/2, n/2)$ for all $p \geq 2$. 
\end{Proposition}

\subsection{Induced subgraph}
Let $W$ be a subset of $V(G)$. 
We define the {\em induced subgraph} $G_{W}$ of $G$ on $W$ as follows:  
\begin{itemize}
	\item $V(G_{W}) = W$. 
	\item $E(G_{W}) = \{ \{u, v\} \in E(G) \mid u, v \in W \}$. 
\end{itemize}
For a vertex $v \in V(G)$, we denote $G_{V(G) \setminus \{v\}}$ by $G \setminus v$. 

\begin{Proposition}\label{G setminus v}
Let $G$ be a finite simple graph on the vertex set $V(G)$ and $v \in V(G)$. Then 
\begin{enumerate}
	\item[$(1)$] ${\rm{ind}}\mathchar`-{\rm{match}}(G \setminus v) \leq {\rm{ind}}\mathchar`-{\rm{match}}(G)$. 
	\item[$(2)$] ${\rm{min}}\mathchar`-{\rm{match}}(G \setminus v) \leq {\rm{min}}\mathchar`-{\rm{match}}(G)$. 
	\item[$(3)$] ${\rm{match}}(G \setminus v) \leq {\rm{match}}(G)$. 
\end{enumerate}
\end{Proposition}
\begin{proof}
(1), (3) : it is easy to prove since induced matchings (resp. matchings) of $G \setminus \{v\}$ are induced matching (resp. matching) of $G$. \\
\ \ (2) : Let $M \subset E(G)$ be a maximal matching with $|M| = \text{min-match}(G)$. 
Note that $V(G) \setminus V(M)$ is an independent set of $G$. 
\begin{itemize}
	\item Assume that $v \in V(G) \setminus V(M)$. 
	Then $\{V(G) \setminus V(M)\} \setminus \{v\}$ is an independent set of $G \setminus v$. 
	Hence $M$ is a maximal matching of $G \setminus v$ since 
	$V(G \setminus v) \setminus V(M) = \{V(G) \setminus V(M)\} \setminus \{v\}$. 
	Thus we have 
	$\text{min-match}(G \setminus v) \leq |M| = \text{min-match}(G)$. 
	\item Assume that $v \not\in V(G) \setminus V(M)$. Then $v \in V(M)$ and there exists an edge $\{v, w\} \in M$.  
	\begin{itemize}
		\item Assume that there exists a vertex $w' \in V(G) \setminus V(M)$ such that $\{ w, w' \} \in E(G)$. Then $\left\{ V(G) \setminus V(M) \right\} \setminus \{w'\}$ is an independent set of $G \setminus v$. 
		Now we put $M_{1} = (M \setminus \{v, w\}) \cup \{w, w'\}$. Since $V(G \setminus v) \setminus V(M_{1}) = \left\{ V(G) \setminus V(M) \right\} \setminus \{w'\}$, it follows that $M_{1}$ is a maximal matching of  $G \setminus v$. Hence we have $\text{min-match}(G \setminus v) \leq |M_{1}| = |M| = \text{min-match}(G)$. 
		\item Assume that $\{w, w''\} \not\in E(G)$ for all $w'' \in V(G) \setminus V(M)$. 
		Then $\{ V(G) \setminus V(M) \} \cup \{w\}$ is an independent set of $G \setminus v$. 
		Put $M_{2} = M \setminus \{v, w\}$. Then $M_{2}$ is a maximal matching of $G \setminus v$ since $V(G \setminus v) \setminus V(M_{2}) = \{ V(G) \setminus V(M) \} \cup \{w\}$. Hence one has $\text{min-match}(G \setminus v) \leq |M_{2}| = |M| - 1 < \text{min-match}(G)$.  
	\end{itemize}
\end{itemize}
Therefore we have the desired conclusion. 
\end{proof}

As a corollary of \ref{G setminus v}, one has

\begin{Corollary}\label{induced-subgraph}
Let $G$ be a finite simple graph on the vertex set $V(G)$ and let $W \subset V(G)$ be a subset. 
Then one has 
\begin{enumerate}
	\item[$(1)$] ${\rm ind}$-${\rm match}(G_{W}) \leq {\rm ind}$-${\rm match}(G)$.
	\item[$(2)$] ${\rm min}$-${\rm match}(G_{W}) \leq {\rm min}$-${\rm match}(G)$.  
	\item[$(3)$] ${\rm match}(G_{W}) \leq {\rm match}(G)$.  
\end{enumerate}
\end{Corollary}

\begin{Lemma}\label{disconnected}
Let $G$ be a finite disconnected simple graph and $G_{1}, \ldots, G_{s}$ $(s \geq 2)$ the connected components of $G$. 
Then we have 
\begin{enumerate}
	\item[$(1)$] $\displaystyle {\rm{ind}}\mathchar`-{\rm{match}}(G) = \sum_{i = 1}^{s} {\rm{ind}}\mathchar`-{\rm{match}}(G_{i})$. 
	\item[$(2)$] $\displaystyle {\rm{min}}\mathchar`-{\rm{match}}(G) = \sum_{i = 1}^{s} {\rm{min}}\mathchar`-{\rm{match}}(G_{i})$.
	\item[$(3)$] $\displaystyle {\rm{match}}(G) = \sum_{i = 1}^{s} {\rm{match}}(G_{i})$.
\end{enumerate}
\end{Lemma}
\begin{proof}
It is enough to show the case $s = 2$. \\
\ \ (1) : For $1 \leq i \leq 2$, let $M_{i} \subset E(G_{i})$ be an induced matching of $G_{i}$ with $|M_{i}| = \text{ind-match}(G_{i})$. 
Since $M_{1} \cup M_{2}$ is an induced matching of $G$ and $M_{1} \cap M_{2} = \emptyset$, one has 
\[
\text{ind-match}(G_{1}) + \text{ind-match}(G_{2}) = |M_{1}| + |M_{2}| = |M_{1}| \cup |M_{2}| \leq \text{ind-match}(G). 
\]
\ \ Next, we show the opposite inequality. 
Let $M$ be an induced matching of $G$ with $|M| = \text{ind-match}(G)$. 
Since each $e \in M$ is an edge of $G_{1}$ or $G_{2}$ and $E(G_{1}) \cap E(G_{2}) = \emptyset$, there exist $M_{1} \subset E(G_{1})$ and $M_{2} \subset E(G_{2})$ such that $M_{1} \cup M_{2} = M$ and $M_{1} \cap M_{2} = \emptyset$.  
Note that $M_{1}$ (resp. $M_{2}$) is an induced matching of $G_{1}$ (resp. $G_{2}$). 
Hence it follows that $|M_{i}| \leq \text{ind-match}(G_{i})$ for $i = 1, 2$. 
Thus one has
\[
\text{ind-match}(G) = |M| = |M_{1} \cup M_{2}| = |M_{1}| + |M_{2}| \leq \text{ind-match}(G_{1}) + \text{ind-match}(G_{2}). 
\] 
Therefore we have the desired conclusion. \\ 
\ \ (2) : Let $M'_{i} \subset E(G_{i})$ be a maximal matching of $G_{i}$ with $|M'_{i}| = \text{min-match}(G_{i})$ for $i = 1, 2$. 
Then $M'_{1} \cap M'_{2} = \emptyset$ and $V(G_{i}) \setminus V(M'_{i})$ is an independent set of $G_{i}$ for $i = 1, 2$. 
Since $G_{1}$ and $G_{2}$ are connected components of $G$, it follows that $\{ V(G_{1}) \setminus V(M'_{1}) \} \cup \{ V(G_{2}) \setminus V(M'_{2}) \}$ is an independent set of $G$. 
Hence $M'_{1} \cup M'_{2}$ is a maximal matching of $G$ since $V(G) \setminus V(M'_{1} \cup M'_{2}) = \{ V(G_{1}) \setminus V(M'_{1}) \} \cup \{ V(G_{2}) \setminus V(M'_{2}) \}$ . 
Thus we have 
\[
\text{min-match}(G_{1}) + \text{min-match}(G_{2}) = |M'_{1}| + |M'_{2}| = |M'_{1} \cup M'_{2}| \geq \text{min-match}(G). 
\]
\ \ Next, we show the opposite inequality. Let $M'$ be a maximal matching of $G$ with $|M'| = \text{min-match}(G)$. Since each $e \in M'$ is an edge of $G_{1}$ or $G_{2}$ and $E(G_{1}) \cap E(G_{2}) = \emptyset$, there exist $M'_{1} \subset E(G_{1})$ and $M'_{2} \subset E(G_{2})$ such that $M' = M'_{1} \cup M'_{2}$ and $M'_{1} \cap M'_{2} = \emptyset$. 
Note that $M'_{i}$ is a matching of $G_{i}$ for $i = 1, 2$.  
Assume that $M'_{1}$ is not maximal. Then there exists $e' \in E(G_{1}) \setminus M'_{1}$ such that $M'_{1} \cup \{e'\}$ is a matching of $G_{1}$. 
However this is a contradiction because $M' \cup \{e'\}$ is a matching of $G$ in this situation. 
Hence we have $M'_{1}$ is a maximal matching of $G_{1}$. Similarly, we also have $M'_{2}$ is a maximal matching of $G_{2}$. 
Thus it follows that 
\[
\text{min-match}(G_{1}) + \text{min-match}(G_{2}) \leq |M'_{1}| + |M'_{2}| = |M'_{1} \cup M'_{2}| = \text{min-match}(G). 
\]
Therefore we have the desired conclusion. \\
\ \ (3) : We can prove this by replacing ``an induced matching" with ``a matching" and ``$\text{ind-match}$" with ``$\text{match}$" in the proof of (1). 
\end{proof}

\begin{Lemma}\label{leaf-edge}
Let $G = (V(G), E(G))$ be a finite simple graph.  
Assume that there exist two edges $\{u, w\}, \{v, w\} \in E(G)$ such that $\deg(u) = \deg(v) = 1$. 
Then 
\begin{enumerate}
	\item[$(1)$] ${\rm ind}$-${\rm match}(G \setminus v) = {\rm ind}$-${\rm match}(G)$.
	\item[$(2)$] ${\rm min}$-${\rm match}(G \setminus v) = {\rm min}$-${\rm match}(G)$.  
	\item[$(3)$] ${\rm match}(G \setminus v) = {\rm match}(G)$.  
\end{enumerate}
\end{Lemma}
\begin{proof}
By virtue of Corollary \ref{induced-subgraph}, it is enough to show $``\geq"$ since $G \setminus v$ is an induced subgraph of $G$. \\
\ \ (1) : Let $M$ be an induced matching of $G$ with $|M| = \text{ind-match}(G)$. 
\begin{itemize}
	\item Assume that $\{v, w\} \not\in M$. Then $M$ is also an induced matching of $G \setminus v$. Hence one has $\text{ind-match}(G \setminus v) \geq |M| = \text{ind-match}(G)$. 
	\item Assume that $\{v, w\} \in M$. Then $\left(M \setminus \{v, w\}\right) \setminus \{u, w\}$ is an induced matching of $G \setminus v$. Hence one has $\text{ind-match}(G \setminus v) \geq |\left(M \setminus \{v, w\}\right) \setminus \{u, w\}| = |M| = \text{ind-match}(G)$. 
\end{itemize}
\ \ (2) : Let $M'$ be a maximal matching of $G \setminus v$ with $M' = \text{min-match}(G \setminus v)$. 
If $w \not\in V(M')$, then $u \not\in V(M')$ since $u$ is only adjacent to $w$. Hence $M' \cup \{u, w\}$ is a matching of $G \setminus v$, but this contradicts the maximality of $M'$. 
Thus $w \in V(M')$ and $M'$ is a maximal matching of $G$. Therefore we have $\text{min-match}(G \setminus v) \geq |M'| = \text{min-match}(G)$. \\
\ \ (3) : We can prove this by replacing ``an induced matching" with ``a matching" and ``$\text{ind-match}$" with ``$\text{match}$" in the proof of (1). 
\end{proof}

\subsection{Special families of connected simple graphs $G_{a, b, c}^{(1)}, \ G_{a, b, c, d, e}^{(2)}$ and $G_{a, b, c}^{(3)}$\ }

In this subsection, we introduce three families of connected simple graphs $G_{a, b, c}^{(1)}, \ G_{a, b, c, d, e}^{(2)}$ and $G_{a, b, c}^{(3)}$. 
These graphs play an important role in the proof of main results.   

First, we introduce the graph $G_{a, b, c}^{(1)}$. \\

\noindent \underline{$G_{a, b, c}^{(1)} : $} \ Let $a, b, c$ be integers with $a \geq 1$, $a \geq b \geq 0$ and $c \geq 0$ and set
\[
X = \{x_{1}, x_{2}, \ldots, x_{2a}\}, \ Y = \{y_{1}, y_{2}, \ldots, y_{2b}\}, \ Z = \{z_{1}, z_{2}, \ldots, z_{c}\}. 
\]
Note that we consider $Y = \emptyset$ if $b = 0$ and $Z = \emptyset$ if $c = 0$. 
We define the graph $G_{a, b, c}^{(1)}$ as follows; see Figure \ref{fig:G1}: 

\begin{itemize}
	\item $V\left( G_{a, b, c}^{(1)} \right) = X \cup Y \cup Z$, 
	\item $\displaystyle E\left( G_{a, b, c}^{(1)} \right) = \left\{ \bigcup_{1 \leq i < j \leq 2a}\{x_{i}, x_{j}\} \right\} \cup \left\{ \bigcup_{i = 1}^{2b} \{x_{i}, y_{i}\}  \right\} \cup \left\{ \bigcup_{i = 1}^{c} \{x_{2a}, z_{i}\} \right\}$. 
\end{itemize}

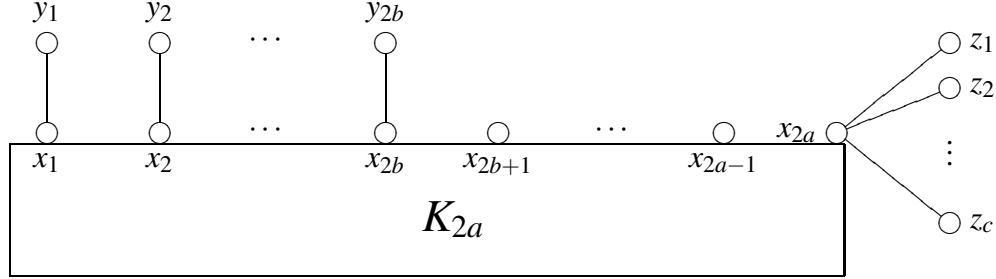
\begin{figure}[htbp]
\centering

\begin{xy}
	\ar@{} (0,0);(20, -16)  *++! U{x_{1}} *\cir<4pt>{} = "X1"
	\ar@{} (0,0);(35, -16)  *++! U{x_{2}} *\cir<4pt>{} = "X2"
	\ar@{} (0,0);(65, -16)  *++! U{x_{2b}} *\cir<4pt>{} = "X2b"
	\ar@{-} "X1";(20, -4) *++! D{y_{1}} *\cir<4pt>{} = "Y1";
	\ar@{-} "X2";(35, -4) *++!D{y_{2}} *\cir<4pt>{} = "Y2";
	\ar@{} "X1"; (49, 0) *++!U{\cdots}
	\ar@{} "X1"; (49, -12) *++!U{\cdots}
	\ar@{-} "X2b";(65, -4) *++!D{y_{2b}} *\cir<4pt>{} = "Y2b";
	\ar@{} (0,0);(80, -16)  *++! U{x_{2b + 1}} *\cir<4pt>{} = "X2b+1";
	\ar@{} (0,0);(110, -16)  *++! U{x_{2a - 1}} *\cir<4pt>{} = "X2a-1";
	\ar@{} (0,0);(125, -16)  *++! R{x_{2a}} *\cir<4pt>{} = "X2a";
	\ar@{-} "X2a";(140, -4) *++! L{z_{1}} *\cir<4pt>{} = "Z1";
	\ar@{-} "X2a";(140, -10) *++!L{z_{2}} *\cir<4pt>{} = "Z2";
	\ar@{-} "X2a";(140, -28) *++!L{z_{c}} *\cir<4pt>{} =  "Zc";
	\ar@{} "X1"; (140, -23) *++!D{\vdots}
	\ar@{} "X1"; (95, -12) *++!U{\cdots}
	\ar@{-} (15,-17.5); (126, -17.5);
	\ar@{-} (15,-17.5); (15, -35);
	\ar@{-} (15,-35); (126, -35);
	\ar@{-} (126,-17.5); (126, -35)
	\ar@{} (0,0);(74,-27.5) *{\Large{\text{$K_{2a}$}}};
\end{xy}

  \caption{The graph $G_{a, b, c}^{(1)}$}
  \label{fig:G1}
\end{figure}

\begin{Lemma}\label{G1}
Let $G_{a, b, c}^{(1)}$ be the graph as above. Then we have 
\begin{enumerate}
	\item[$(1)$] $\left| V\left(G_{a, b, c}^{(1)}\right) \right| = 2a + 2b + c$. 
	\item[$(2)$] $\text{ind-match}\left(G_{a, b, c}^{(1)}\right) = 1$. 
	\item[$(3)$] $\text{min-match}\left(G_{a, b, c}^{(1)}\right) = a$. 
	\item[$(4)$] $\text{match}\left(G_{a, b, c}^{(1)}\right) = a + b$. 
\end{enumerate}
\end{Lemma}
\begin{proof}
(1) : \ $\left| V\left(G_{a, b, c}^{(1)}\right) \right| = |X \cup Y \cup Z| = 2a + 2b + c$. \\
\ \ (2) : Since each edge of $G_{a, b, c}^{(1)}$ contains a vertex of the complete subgraph $K_{2a}$, there is no induced matching $M$ of $G_{a, b, c}^{(1)}$ with $|M| \geq 2$. 
Hence one has $\text{ind-match}\left(G_{a, b, c}^{(1)}\right) = 1$. \\
\ \ (3) : By using Corollary \ref{induced-subgraph}, it follows that  
\[
\text{min-match}\left(G_{a, b, c}^{(1)}\right) \geq \text{min-match}\left( \left(G_{a, b, c}^{(1)}\right)_{X} \right) = \text{min-match}(K_{2a}) = a. 
\]
Moreover, it also follows that $\text{min-match}\left(G_{a, b, c}^{(1)}\right) \leq a$ 
since $\displaystyle \bigcup_{i = 1}^{a}\{x_{i}, x_{a + i}\}$ is a maximal matching of 
$\left(G_{a, b, c}^{(1)}\right)$. 
Thus we have $\text{min-match}\left(G_{a, b, c}^{(1)}\right) = a$. \\
\ \ (4) : Let $\displaystyle M = \left\{ \bigcup_{i = 1}^{2b} \{x_{i}, y_{i}\}  \right\} \cup \left\{ \bigcup_{i = 1}^{a - b} \{x_{2b + i}, x_{a + b + i}\}  \right\}$. 
Then one has $\text{match}\left(G_{a, b, c}^{(1)}\right) \geq a + b$ since $M$ is a matching of $G_{a, b, c}^{(1)}$ with $|M| = a + b$.
If $Z = \emptyset$, then $\text{match}\left(G_{a, b, c}^{(1)}\right) = a + b$ holds since $M$ is a perfect matching of $G_{a, b, c}^{(1)}$. 

Assume that $Z \neq \emptyset$. 
Then, by virtue of Proposition \ref{important}(2) together with Lemma \ref{leaf-edge}, it follows that
\[
\text{match}\left(G_{a, b, c}^{(1)}\right) = \text{match}\left(G_{a, b, 1}^{(1)}\right) \leq \left\lfloor \frac{\left| G_{a, b, 1}^{(1)} \right|}{2} \right\rfloor = \left\lfloor \frac{2a + 2b + 1}{2} \right\rfloor = a + b. 
\]
Hence we have $\text{match}\left(G_{a, b, c}^{(1)}\right) = a + b$. 
\end{proof}

\ 

Next, we introduce the graph $G_{a, b, c, d, e}^{(2)}$. \\

\noindent \underline{$G_{a, b, c, d, e}^{(2)} : $} \ Let $a, b, c ,d, e$ be integers with $a > b \geq 0$, $c \geq 1$, $d, e \geq 0$ and $d + e \geq 1$ and set
\[
X = \{x_{1}, x_{2}, \ldots, x_{2a}\}, \ Y = \{y_{1}, y_{2}, \ldots, y_{2b}\}, \ Z = \{z_{1}, z_{2}, \ldots, z_{c}\},
\]
\[
U = \{u_{1}, u_{2}, \ldots, u_{2d}\}, \ U' = \{u'_{1}, u'_{2}, \ldots, u'_{2d}\}, \ V = \{v_{1}, v_{2}, \ldots, v_{2e}\}. 
\] 
Note that we consider $Y = \emptyset$ if $b = 0$, $U = U' = \emptyset$ if $d = 0$ and $V = \emptyset$ if $e = 0$. 
We define the graph $G_{a, b, c, d, e}^{(2)}$ as follows; see Figure \ref{fig:G3}: 

\begin{itemize}
	\item $V\left( G_{a, b, c, d, e}^{(2)} \right) = X \cup Y \cup Z \cup U \cup U' \cup V \cup \{w\}$, 
	\item $\displaystyle E\left( G_{a, b, c, d, e}^{(2)} \right) \\ = \left\{ \bigcup_{1 \leq i < j \leq 2a}\{x_{i}, x_{j}\} \right\} \cup \left\{ \bigcup_{i = 1}^{2b} \{x_{i}, y_{i}\}  \right\} \cup \left\{ \bigcup_{i = 1}^{c} \{x_{2a}, z_{i}\} \right\} \cup \left\{ \bigcup_{i = 1}^{d} \{u_{i}, u_{d + i}\} \right\}$ \\ 
	\ \\
	\[
	\cup \left\{ \bigcup_{i = 1}^{d} \{u_{i}, u'_{i}\} \right\} \cup \left\{ \bigcup_{i = 1}^{e} \{v_{i}, v_{e + i}\} \right\} \cup \left\{ \{w, w'\} \mid w' \in X \cup V \cup U \right\}.  
	\]
\end{itemize}

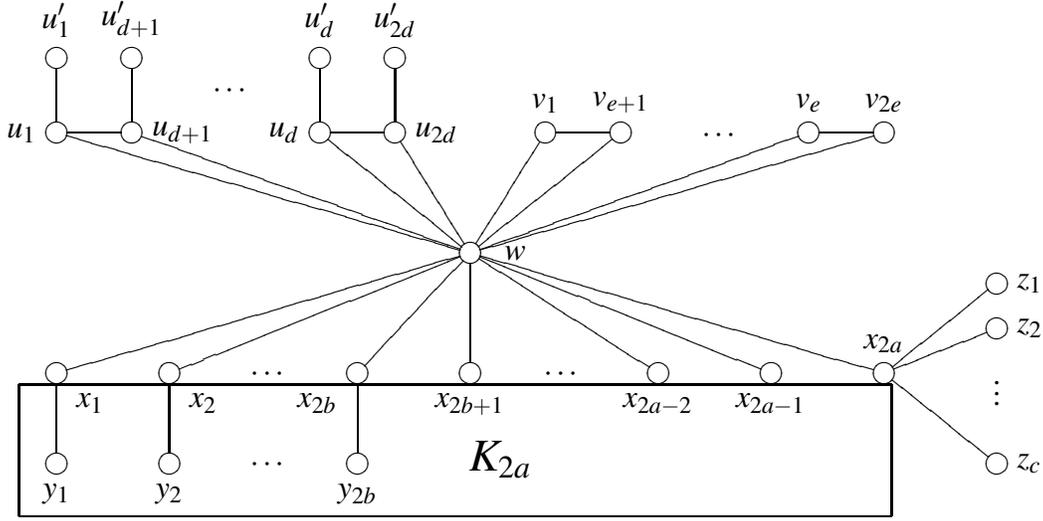
\begin{figure}[htbp]
\centering

\begin{xy}
	\ar@{} (0,0);(15, -16)  *++! LU{x_{1}} *\cir<4pt>{} = "X1"
	\ar@{} (0,0);(30, -16)  *++! LU{x_{2}} *\cir<4pt>{} = "X2"
	\ar@{} (0,0);(55, -16)  *++! RU{x_{2b}} *\cir<4pt>{} = "X2b"
	\ar@{-} "X1";(15, -28) *++! U{y_{1}} *\cir<4pt>{} = "Y1";
	\ar@{-} "X2";(30, -28) *++! U{y_{2}} *\cir<4pt>{} = "Y2";
	\ar@{} "X1"; (43, -24.5) *++!U{\cdots}
	\ar@{} "X1"; (43, -12.5) *++!U{\cdots}
	\ar@{-} "X2b";(55, -28) *++!U{y_{2b}} *\cir<4pt>{} = "Y2b";
	\ar@{} (0,0);(70, -16)  *++! U{x_{2b + 1}} *\cir<4pt>{} = "X2b+1";
	\ar@{} (0,0);(95, -16)  *++! U{x_{2a - 2}} *\cir<4pt>{} = "X2a-2";
	\ar@{} (0,0);(110, -16)  *++! U{x_{2a - 1}} *\cir<4pt>{} = "X2a-1";
	\ar@{} (0,0);(125, -16)  *++! D{x_{2a}} *\cir<4pt>{} = "X2a";
	\ar@{-} "X2a";(140, -4) *++! L{z_{1}} *\cir<4pt>{} = "Z1";
	\ar@{-} "X2a";(140, -10) *++!L{z_{2}} *\cir<4pt>{} = "Z2";
	\ar@{-} "X2a";(140, -28) *++!L{z_{c}} *\cir<4pt>{} =  "Zc";
	\ar@{} "X1"; (140, -23) *++!D{\vdots}
	\ar@{} "X1"; (82, -12.5) *++!U{\cdots}
	\ar@{-} (10,-17.5); (126, -17.5);
	\ar@{-} (10,-17.5); (10, -35);
	\ar@{-} (10,-35); (126, -35);
	\ar@{-} (126,-17.5); (126, -35)
	\ar@{} (0,0);(74,-27.5) *{\Large{\text{$K_{2a}$}}};
	\ar@{-} "X1";(70, 0) *\cir<4pt>{} = "W";
	\ar@{} (0,0);(76,0) *{\text{$w$}};
	\ar@{-} "X2";"W";
	\ar@{-} "X2b";"W";
	\ar@{-} "X2b+1";"W";
	\ar@{-} "X2a-2";"W";
	\ar@{-} "X2a-1";"W";
	\ar@{-} "X2a";"W";
	\ar@{-} "W";(15, 16)  *++! R{u_{1}} *\cir<4pt>{} = "U1";
	\ar@{-} "W";(25, 16)  *++! L{u_{d + 1}} *\cir<4pt>{} = "Ud+1";
	\ar@{-} "W";(50, 16)  *++! R{u_{d}} *\cir<4pt>{} = "Ud";
	\ar@{-} "W";(60, 16)  *++! L{u_{2d}} *\cir<4pt>{} = "U2d";
	\ar@{-} "W";(80, 16)  *++! D{v_{1}} *\cir<4pt>{} = "V1";
	\ar@{-} "W";(90, 16)  *++! D{v_{e + 1}} *\cir<4pt>{} = "Ve+1";
	\ar@{-} "W";(115, 16)  *++! D{v_{e}} *\cir<4pt>{} = "Ve";
	\ar@{-} "W";(125, 16)  *++! D{v_{2e}} *\cir<4pt>{} = "V2e";
	\ar@{-} "U1";"Ud+1";
	\ar@{-} "Ud";"U2d";
	\ar@{-} "V1";"Ve+1";
	\ar@{-} "Ve";"V2e";
	\ar@{-} "U1";(15, 26)  *++! D{u'_{1}} *\cir<4pt>{} = "U'1";
	\ar@{-} "Ud+1";(25, 26)  *++! D{u'_{d + 1}} *\cir<4pt>{} = "U'd+1";
	\ar@{-} "Ud";(50, 26)  *++! D{u'_{d}} *\cir<4pt>{} = "U'd";
	\ar@{-} "U2d";(60, 26)  *++! D{u'_{2d}} *\cir<4pt>{} = "U'2d";
	\ar@{} "X1"; (38, 18) *++!D{\cdots}
	\ar@{} "X1"; (103, 12) *++!D{\cdots}
\end{xy}

  \caption{The graph $G_{a, b, c, d, e}^{(2)}$}
  \label{fig:G3}
\end{figure}

\begin{Lemma}\label{G2}\normalfont
Let $G_{a, b, c, d, e}^{(2)}$ be the graph as above. Then we have 
\begin{enumerate}
	\item $\left| V\left(G_{a, b, c, d, e}^{(2)}\right) \right| = 2a + 2b + c + 4d + 2e + 1$. 
	\item $\text{ind-match}\left(G_{a, b, c, d, e}^{(2)}\right) = d + e + 1$. 
	\item $\text{min-match}\left(G_{a, b, c, d, e}^{(2)}\right) = a + d + e$. 
	\item $\text{match}\left(G_{a, b, c, d, e}^{(2)}\right) = a + b + 2d + e + 1$. 
\end{enumerate}
\end{Lemma}
\begin{proof}
(1) : $\left| V\left(G_{a, b, c, d, e}^{(2)}\right) \right| = |X \cup Y \cup Z \cup U \cup U' \cup V \cup \{w\}| =  2a + 2b + c + 4d + 2e + 1$. \\
To prove (2), (3) and (4), we calculate the induced matching number and the minimum matching number of the induced subgraphs $\left\{ G_{a, b, c, d, e}^{(2)} \right\}_{X \cup Y \cup Z}$, $\left\{ G_{a, b, c, d, e}^{(2)} \right\}_{U \cup U'}$ and $\left\{ G_{a, b, c, d, e}^{(2)} \right\}_{V}$.  
\begin{itemize}
	\item Since $\left\{ G_{a, b, c, d, e}^{(2)} \right\}_{X \cup Y \cup Z} = G_{a, b, c}^{(1)}$ , by virtue of Lemma \ref{G1}, it follows that 
	\begin{itemize}
		\item $\displaystyle \text{ind-match}\left( \left\{ G_{a, b, c, d, e}^{(2)} \right\}_{X \cup Y \cup Z} \right) = 1$. 
		\item $\displaystyle \text{min-match}\left( \left\{ G_{a, b, c, d, e}^{(2)} \right\}_{X \cup Y \cup Z} \right) = a$. \\
	\end{itemize}
	\item For each $1 \leq i \leq d$, let $U_{i} = \{u_{i}, u_{d + i}, u'_{i}, u'_{d + i}\}$. Then  
	$\left\{ G_{a, b, c, d, e}^{(2)} \right\}_{U_{i}} = P_{4}$ and 
	\[
	\left\{ G_{a, b, c, d, e}^{(2)} \right\}_{U \cup U'} = \bigcup_{i = 1}^{d} \left\{ G_{a, b, c, d, e}^{(2)} \right\}_{U_{i}} = dP_{4}, 
	\]
	where $P_{4}$ is the path graph with $|V(P_{4})| = 4$  and $dP_{4}$ is the disjoint union of  $d$ copies of $P_{4}$. 
	Since $\text{ind-match}(P_{4}) = \text{min-match}(P_{4}) = 1$, by virtue of Lemma  \ref{disconnected}, it follows that 
	\begin{itemize}
		\item $\displaystyle \text{ind-match}\left( \left\{ G_{a, b, c, d, e}^{(2)} \right\}_{U \cup U'} \right) = \text{ind-match}(dP_{4}) = d \cdot \text{ind-match}(P_{4}) = d$. 
		\item $\displaystyle \text{min-match}\left( \left\{ G_{a, b, c, d, e}^{(2)} \right\}_{U \cup U'} \right) = \text{min-match}(dP_{4}) = d \cdot \text{min-match}(P_{4}) = d$. \\
	\end{itemize}
	\item Since  $\left\{ G_{a, b, c, d, e}^{(2)} \right\}_{V} = eK_{2}$ and $\text{ind-match}(K_{2}) = \text{min-match}(K_{2}) = 1$, by virtue of Lem \ref{disconnected}, it follows that 
	\begin{itemize}
		\item $\displaystyle \text{ind-match}\left( \left\{ G_{a, b, c, d, e}^{(2)} \right\}_{V} \right) = \text{ind-match}(eP_{2}) = e \cdot \text{ind-match}(P_{2}) = e$. 
		\item $\displaystyle \text{min-match}\left( \left\{ G_{a, b, c, d, e}^{(2)} \right\}_{V} \right) = \text{min-match}(eP_{2}) = e \cdot \text{min-match}(P_{2}) = e$. \\
	\end{itemize}
\end{itemize}
Now we are in position to prove (2), (3) and (4). \\

\ \ (2) : Let $S = Y \cup Z \cup U'$. Then $S$ is an independent set of $G_{a, b, c, d, e}^{(2)} \setminus w$ and $G_{a, b, c, d, e}^{(2)}$ is the $S$-suspension of $G_{a, b, c, d, e}^{(2)} \setminus w$. 
Since $G_{a, b, c, d, e}^{(2)} \setminus w$ is the disjoint union of $\left\{ G_{a, b, c, d, e}^{(2)} \right\}_{X \cup Y \cup Z}$, $\left\{ G_{a, b, c, d, e}^{(2)} \right\}_{U \cup U'}$ and $\left\{ G_{a, b, c, d, e}^{(2)} \right\}_{V}$, by virtue of Lemmas \ref{S-suspension}, \ref{disconnected} and the above calculation, it follows that  
\begin{eqnarray*}
& & \text{ind-match}\left(G_{a, b, c, d, e}^{(2)}\right) \\
&=& \text{ind-match}\left(G_{a, b, c, d, e}^{(2)} \setminus w\right) \\
&=& \text{ind-match}\left( \left\{ G_{a, b, c, d, e}^{(2)} \right\}_{X \cup Y \cup Z} \cup \left\{ G_{a, b, c, d, e}^{(2)} \right\}_{U \cup U'} \cup \left\{ G_{a, b, c, d, e}^{(2)} \right\}_{V} \right) \\
&=& \text{ind-match}\left( \left\{ G_{a, b, c, d, e}^{(2)} \right\}_{X \cup Y \cup Z} \right) + \text{ind-match}\left( \left\{ G_{a, b, c, d, e}^{(2)} \right\}_{U \cup U'} \right) \\ 
& & + \text{ind-match}\left( \left\{ G_{a, b, c, d, e}^{(2)} \right\}_{V} \right) \\
&=& 1 + d + e. 
\end{eqnarray*}
\ \ (3) : Since $G_{a, b, c, d, e}^{(2)} \setminus w$ is the disjoint union of $\left\{ G_{a, b, c, d, e}^{(2)} \right\}_{X \cup Y \cup Z}$, $\left\{ G_{a, b, c, d, e}^{(2)} \right\}_{U \cup U'}$ and $\left\{ G_{a, b, c, d, e}^{(2)} \right\}_{V}$, 
by virtue of Corollary \ref{induced-subgraph}, Lemma \ref{disconnected} and the above calculation, it follows that 
\begin{eqnarray*}
& & \text{min-match}\left(G_{a, b, c, d, e}^{(2)}\right) \\
&\geq& \text{min-match}\left(G_{a, b, c, d, e}^{(2)} \setminus w\right) \\
&=& \text{min-match}\left( \left\{ G_{a, b, c, d, e}^{(2)} \right\}_{X \cup Y \cup Z} \cup \left\{ G_{a, b, c, d, e}^{(2)} \right\}_{U \cup U'} \cup \left\{ G_{a, b, c, d, e}^{(2)} \right\}_{V} \right) \\
&=& \text{min-match}\left( \left\{ G_{a, b, c, d, e}^{(2)} \right\}_{X \cup Y \cup Z} \right) + \text{min-match}\left( \left\{ G_{a, b, c, d, e}^{(2)} \right\}_{U \cup U'} \right) \\
& & + \text{min-match}\left( \left\{ G_{a, b, c, d, e}^{(2)} \right\}_{V} \right) \\
&=& a + d + e. 
\end{eqnarray*}
\ \ Next, put $\displaystyle M = \left\{ \bigcup_{i = 1}^{a}\{x_{i}, x_{a + i}\} \right\} \cup \left\{ \bigcup_{i = 1}^{d} \{u_{i}, u_{d + i}\}  \right\} \cup \left\{ \bigcup_{i = 1}^{e} \{v_{i}, v_{e + i}\} \right\}$. 
Then $M$ is a maximal matching with $|M| = a + d + e$ since $\displaystyle V\left(G_{a, b, c, d, e}^{(2)}\right) \setminus V(M) = Y \cup Z \cup U' \cup \{w\}$ is an independent set of $G_{a, b, c, d, e}^{(2)}$. 
Hence $\text{min-match}\left(G_{a, b, c, d, e}^{(2)}\right) \leq |M| = a + d + e$. 
Thus one has $\text{min-match}\left(G_{a, b, c, d, e}^{(2)}\right) = a + d + e$. \\
\ \ (4) : Let 
\begin{eqnarray*}
M' &=& \left\{ \bigcup_{i = 1}^{2b} \{x_{i}, y_{i}\}  \right\} \cup \left\{ \bigcup_{i = 1}^{a - b - 1} \{ x_{2b + i}, x_{a - b - 1 + i} \} \right\} \cup \left\{ \bigcup_{i = 1}^{2d} \{u_{i}, u'_{i}\} \right\} \\ 
& & \cup \left\{ \bigcup_{i = 1}^{e} \{v_{i}, v_{e + i}\} \right\} \cup \{x_{2a - 1}, w\} \cup \{x_{2a}, z_{1}\}. 
\end{eqnarray*} 
Then we have $\text{match}\left(G_{a, b, c, d, e}^{(2)}\right) \geq |M'| = a + b + 2d + e + 1$ 
since $M'$ is a matching of $G_{a, b, c, d, e}^{(2)}$ with $|M| = 2b + (a - b - 1) + 2d + e + 2 = a + b + 2d + e + 1$. 
Moreover, by virtue of Proposition \ref{important}(2) together with Lemma \ref{leaf-edge}(3), it follows that 
\begin{eqnarray*}
\text{match}\left(G_{a, b, c, d, e}^{(2)}\right) &=& \text{match}\left( G_{a, b, 1, d, e}^{(2)}\right) \\
&\leq& \left\lfloor \frac{\left| G_{a, b, 1, d, e}^{(2)} \right|}{2} \right\rfloor \\
&=& \left\lfloor \frac{ 2a + 2b + 4d + 2e + 2}{2} \right\rfloor \\
&=& a + b + 2d + e + 1. 
\end{eqnarray*}
Therefore we have $\text{match}\left(G_{a, b, c, d, e}^{(2)}\right) = a + b + 2d + e + 1$.  
\end{proof}

\ 

\ \ Finally, we introduce the graph $G_{a, b, c}^{(3)}$ .  \\

\noindent \underline{$G_{a, b, c}^{(3)} : $} \ Let $a, b, c$ be integers with $a \geq 1$, $b \geq0$ and $c \geq 1$ and set 
\[
X = \{x_{1}, x_{2}, \ldots, x_{2a}\}, \ Y = \{y_{1}, y_{2}, \ldots, y_{2b}\}, \ Z = \{z_{1}, z_{2}, \ldots, z_{c}\}.   
\]
Note that $Y = \emptyset$ if $b = 0$. 
We define the graph $G_{a, b, c}^{(3)}$ as follows; see Figure \ref{fig:G4}: 
\begin{itemize}
	\item $V\left( G_{a, b, c}^{(3)} \right) = X \cup Y \cup Z \cup \{v\} \cup \{w\}$, 
	\item $\displaystyle E\left( G_{a, b, c}^{(3)} \right) = \left\{ \bigcup_{1 \leq i < j \leq 2a}\{x_{i}, x_{j}\} \right\} \cup \left\{ \bigcup_{i = 1}^{b} \{y_{i}, y_{b + i}\}  \right\} \cup \left\{ \bigcup_{i = 1}^{c} \{v, z_{i}\} \right\}$ \\
	\ \\  
	$\ \ \ \ \ \ \ \ \ \ \ \ \ \ \ \ \ \ \cup \left\{ \{w, w'\} \mid w' \in X \cup Y \cup \{v\} \right\}$. 
\end{itemize}

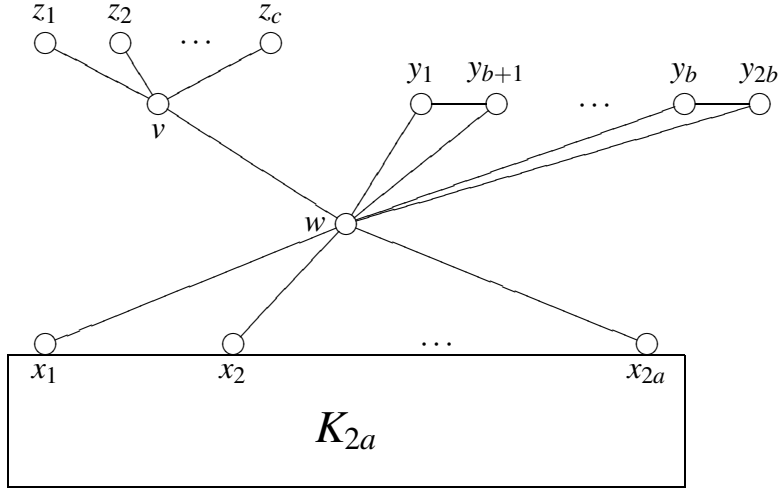
\begin{figure}[htbp]
\centering
\bigskip

\begin{xy}
	\ar@{} (0,0);(30, -16)  *++! U{x_{1}} *\cir<4pt>{} = "X1"
	\ar@{} (0,0);(55, -16)  *++! U{x_{2}} *\cir<4pt>{} = "X2"
	\ar@{} (0,0);(110, -16)  *++! U{x_{2a}} *\cir<4pt>{} = "X2a";
	\ar@{} "X1"; (82, -12.5) *++!U{\cdots}
	\ar@{-} (25,-17.5); (115, -17.5);
	\ar@{-} (25,-17.5); (25, -35);
	\ar@{-} (25,-35); (115, -35);
	\ar@{-} (115,-17.5); (115, -35)
	\ar@{} (0,0);(70,-27.5) *{\Large{\text{$K_{2a}$}}};
	\ar@{-} "X1";(70, 0) *++! R{w} *\cir<4pt>{} = "W";
	\ar@{-} "X2";"W";
	\ar@{-} "X2a";"W";
	\ar@{-} "W";(45, 16)  *++! U{v} *\cir<4pt>{} = "V";
	\ar@{-} "V";(30, 24) *++! D{z_{1}} *\cir<4pt>{} = "Z1";
	\ar@{-} "V";(40, 24) *++!D{z_{2}} *\cir<4pt>{} = "Z2";
	\ar@{-} "V";(60, 24) *++!D{z_{c}} *\cir<4pt>{} =  "Zc";
	\ar@{-} "W";(80, 16)  *++! D{y_{1}} *\cir<4pt>{} = "Y1";
	\ar@{-} "W";(90, 16)  *++! D{y_{b + 1}} *\cir<4pt>{} = "Yb+1";
	\ar@{-} "W";(115, 16)  *++! D{y_{b}} *\cir<4pt>{} = "Yb";
	\ar@{-} "W";(125, 16)  *++! D{y_{2b}} *\cir<4pt>{} = "Y2b";
	\ar@{-} "Y1";"Yb+1";
	\ar@{-} "Yb";"Y2b";
	\ar@{} "X1"; (50, 20.5) *++!D{\cdots}
	\ar@{} "X1"; (103, 12) *++!D{\cdots}
\end{xy}

  \caption{The graph $G_{a, b, c}^{(3)}$}
  \label{fig:G4}
\end{figure}

\begin{Lemma}\label{G3}\normalfont
Let $G_{a, b, c}^{(3)}$ be the graph as above. Then we have 
\begin{enumerate}
	\item $\left| V\left(G_{a, b, c}^{(3)}\right) \right| = 2a + 2b + c + 2$. 
	\item $\text{ind-match}\left(G_{a, b, c}^{(3)}\right) = b + 2$. 
	\item $\text{min-match}\left(G_{a, b, c}^{(3)}\right) = \text{match}\left(G_{a, b, c}^{(3)}\right) = a + b + 1$. 
\end{enumerate}
\end{Lemma}
\begin{proof}
(1) : $\displaystyle \left| V\left(G_{a, b, c}^{(3)}\right) \right| = |X \cup Y \cup Z \cup \{v\} \cup \{w\}| = 2a + 2b + c + 2$. \\
To prove (2) and (3), we calculate the induced matching number and the minimum matching number of the induced subgraph $\left\{ G_{a, b, c}^{(3)} \right\}_{X}$,  
$\left\{ G_{a, b, c}^{(3)} \right\}_{Y}$ and $\left\{ G_{a, b, c}^{(3)} \right\}_{Z \cup \{v\}}$. 
\begin{itemize}
	\item Since $\left\{ G_{a, b, c}^{(3)} \right\}_{X} = K_{2a}$, it follows that 
	\begin{itemize}
		\item $\text{ind-match}\left(\left\{ G_{a, b, c}^{(3)} \right\}_{X}\right) = \text{ind-match}(K_{2a}) = 1$. 
		\item $\text{min-match}\left(\left\{ G_{a, b, c}^{(3)} \right\}_{X}\right) = \text{min-match}(K_{2a}) = a$. \\
	\end{itemize} 
	\item Since $\left\{ G_{a, b, c}^{(3)} \right\}_{Y} = bK_{2}$ and $\text{ind-match}(K_{2}) = \text{min-match}(K_{2}) = 1$, by Lemma \ref{disconnected}, it follows that 
	\begin{itemize}
		\item $\text{ind-match}\left( \left\{ G_{a, b, c}^{(3)} \right\}_{Y} \right) = \text{ind-match}(bP_{2}) = b \cdot \text{ind-match}(P_{2}) = b$. 
		\item $\text{min-match}\left( \left\{ G_{a, b, c}^{(3)} \right\}_{Y} \right) = \text{min-match}(bP_{2}) = b \cdot \text{min-match}(P_{2}) = b$. \\
	\end{itemize}
	\item Since $\left\{ G_{a, b, c}^{(3)} \right\}_{Z \cup \{v\}}$ is the star graph $K_{1, c}$, it follows that 
	\begin{itemize}
		\item $\text{ind-match}\left( \left\{ G_{a, b, c}^{(3)} \right\}_{Z \cup \{v\}} \right) = 1$. 
		\item $\text{min-match}\left( \left\{ G_{a, b, c}^{(3)} \right\}_{Z \cup \{v\}} \right) = 1$. \\ 
	\end{itemize}  
\end{itemize}

Now we are in position to prove (2) and (3). \\

\ \ (2) : Note that $Z$ is an independent set of $G_{a, b, c}^{(3)} \setminus w$ and  
$G_{a, b, c}^{(3)}$ is the $S$-suspension of $G_{a, b, c}^{(3)} \setminus w$. 
Since $G_{a, b, c}^{(3)} \setminus w$ is the disjoint union of $\left\{ G_{a, b, c}^{(3)} \right\}_{X}$, 
$\left\{ G_{a, b, c}^{(3)} \right\}_{Y}$ and $\left\{ G_{a, b, c}^{(3)} \right\}_{Z \cup \{v\}}$, by virtue of 
Lemmas \ref{S-suspension}, \ref{disconnected} and the above calculation, it follows that 
\begin{eqnarray*}
& & \text{ind-match}\left(G_{a, b, c}^{(3)}\right) \\
&=& \text{ind-match}\left(G_{a, b, c}^{(3)} \setminus w \right) \\
&=& \text{ind-match}\left( \left\{ G_{a, b, c}^{(3)} \right\}_{X} \cup \left\{ G_{a, b, c}^{(3)} \right\}_{Y} \cup \left\{ G_{a, b, c}^{(3)} \right\}_{Z \cup \{v\}} \right) \\
&=& \text{ind-match}\left( \left\{ G_{a, b, c}^{(3)} \right\}_{X} \right) + \text{ind-match}\left( \left\{ G_{a, b, c}^{(3)} \right\}_{Y} \right) + \text{ind-match}\left( \left\{ G_{a, b, c}^{(3)} \right\}_{Z \cup \{v\}} \right) \\
&=& 1 + b + 1 \\
&=& b + 2. 
\end{eqnarray*}
\ \ (3) : Since $G_{a, b, c}^{(3)} \setminus w$ is the disjoint union of $\left\{ G_{a, b, c}^{(3)} \right\}_{X}$, $\left\{ G_{a, b, c}^{(3)} \right\}_{Y}$ and $\left\{ G_{a, b, c}^{(3)} \right\}_{Z \cup \{v\}}$, by virtue of Corollary \ref{induced-subgraph}, Lemma \ref{disconnected} and the above calculation, it follows that 
\begin{eqnarray*}
& & \text{min-match}\left(G_{a, b, c}^{(3)}\right) \\ 
&\geq& \text{min-match}\left(G_{a, b, c}^{(3)} \setminus w \right) \\
&=& \text{min-match}\left( \left\{ G_{a, b, c}^{(3)} \right\}_{X} \cup \left\{ G_{a, b, c}^{(3)} \right\}_{Y} \cup \left\{ G_{a, b, c}^{(3)} \right\}_{Z \cup \{v\}} \right) \\
&=& \text{min-match}\left( \left\{ G_{a, b, c}^{(3)} \right\}_{X} \right) + \text{min-match}\left( \left\{ G_{a, b, c}^{(3)} \right\}_{Y} \right) + \text{min-match}\left( \left\{ G_{a, b, c}^{(3)} \right\}_{Z \cup \{v\}} \right) \\
&=& a + b + 1. 
\end{eqnarray*}
Moreover, by Proposition \ref{important} and Lemma \ref{leaf-edge}, one has 
\begin{eqnarray*}
\text{min-match}\left(G_{a, b, c}^{(3)}\right) &\leq& \text{match}\left(G_{a, b, c}^{(3)}\right) \\
&=& \text{match}\left( G_{a, b, 1}^{(3)} \right) \\
&\leq& \left\lfloor \frac{\left| G_{a, b, 1}^{(3)} \right|}{2} \right\rfloor \\
&=& \left\lfloor \frac{ 2a + 2b + 3}{2} \right\rfloor \\
&=& a + b + 1. 
\end{eqnarray*}
Therefore we have $\text{min-match}\left(G_{a, b, c}^{(3)}\right) = \text{match}\left(G_{a, b, c}^{(3)}\right) = a + b + 1$. 
\end{proof}


\section{Proof of the first main result}

In this section, we give a proof of the first main result as below. 

\begin{Theorem}\label{first-main}
Let $n \geq 2$ be an integer and set 
\begin{eqnarray*}
& & {\rm{\bf Graph}_{ind\text{-}match, min\text{-}match, match}}(n) \\
&=& \left\{(p, q, r) \in \mathbb{N}^{3} ~\left|~
\begin{array}{c}
  \mbox{\rm{There\ exists\ a\ connected\ simple\ graph}\ $G$ \ \rm{with}\ $|V(G)| = n$} \\
  \mbox{{\rm{and}} \ $\text{\rm ind-match}(G) = p, \ \text{\rm min-match}(G) = q, \ \text{\rm match}(G) = r$} \\ 
\end{array}
\right \}\right. .  
\end{eqnarray*}
Then we have the following: 
\begin{enumerate}
	\item[$(1)$] If $n$ is odd, then 
	\begin{eqnarray*}
	& & {\rm{\bf Graph}_{ind\text{-}match, min\text{-}match, match}}(n) \\
	&=& \left\{ (p, q, r) \in \mathbb{N}^{3} \ \middle| \ 1 \leq p \leq q \leq r \leq 2q \ \ {\rm{and}} \ \ r \leq \frac{n-1}{2} \right\} . 
	\end{eqnarray*}
	\item[$(2)$] If $n$ is even, then 
	\begin{eqnarray*}
	& & {\rm{\bf Graph}_{ind\text{-}match, min\text{-}match, match}}(n) \\ 
	&=& \left\{ (1, q, r) \in \mathbb{N}^{3} \ \middle| \ 1 \leq q \leq r \leq 2q \ \ {\rm{and}} \ \ r \leq \frac{n}{2} \right\} \\
	&\cup& \left\{ (p, q, r) \in \mathbb{N}^{3} \ \middle| \ 2 \leq p \leq q \leq r \leq 2q, \ r \leq \frac{n}{2}\ \ {\rm{and}}\ \ (q, r) \neq \left( \frac{n}{2}, \frac{n}{2} \right) \right\} . 
	\end{eqnarray*}
\end{enumerate}
\end{Theorem}
\begin{proof}
Assume that $(p, q, r) \in {\rm{Graph}_{ind\text{-}match, min\text{-}match, match}}(n)$. 
Then there exists a connected simple graph $G$ such that $|V(G)| = n$, $\text{\rm ind-match}(G) = p, \ \text{\rm min-match}(G) = q$ and $\text{\rm match}(G) = r$. 
\begin{itemize}
	\item If $n$ is odd, then $\displaystyle \left\lfloor \frac{n}{2} \right\rfloor =  \frac{n - 1}{2}$. Hence, by Proposition \ref{important}, one has
	\[
	(p, q, r) \in \left\{ (p, q, r) \in \mathbb{N}^{3} \ \middle| \ 1 \leq p \leq q \leq r \leq 2q \ \ \text{and} \ \ r \leq \frac{n-1}{2} \right\}. 
	\] 
	\item If $n$ is even, then $\displaystyle \left\lfloor \frac{n}{2} \right\rfloor =  \frac{n}{2}$. 
	If $p = 1$, then we have 
	\[
	(1, q, r) \in \left\{ (1, q, r) \in \mathbb{N}^{3} \ \middle| \ 1 \leq q \leq r \leq 2q \ {\rm{and}} \ r \leq \frac{n}{2} \right\}
	\]
	by Proposition \ref{important}. \\
	Assume that $p \geq 2$. By virtue of Propositions \ref{important} and \ref{min = n/2}, one has
	\[
	(p, q, r) \in \left\{ (p, q, r) \in \mathbb{N}^{3} \ \middle| \ 2 \leq p \leq q \leq r \leq 2q, \ r \leq \frac{n}{2}\ \ {\rm{and}}\ \ (q, r) \neq \left( \frac{n}{2}, \frac{n}{2} \right) \right\}. 
	\] 
\end{itemize}

\ 

We show the converse inclusion. 
Assume that $n$ is odd and 
\[
(p, q, r) \in \left\{ (p, q, r) \in \mathbb{N}^{3} \ \middle| \ 1 \leq p \leq q \leq r \leq 2q \ \ \text{and} \ \ r \leq \frac{n-1}{2} \right\}. 
\] 

\begin{itemize}
	\item Assume $p = 1$. Note that $q \geq 1$, $q \geq r - q \geq 0$ and $n - 2r \geq 1$. 
	Let us consider the graph $\displaystyle G_{q, k, n - 2(q + k)}^{(1)}$, where $k = r - q$; see Figure \ref{fig:G5}: 
		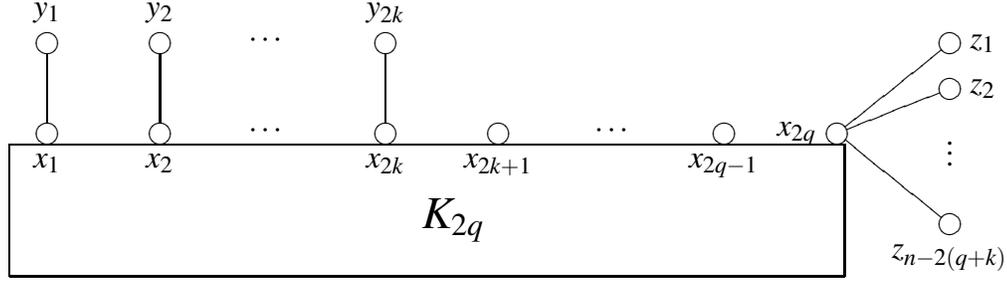
\begin{figure}[htbp]
	\centering

	\begin{xy}
		\ar@{} (0,0);(20, -16)  *++! U{x_{1}} *\cir<4pt>{} = "X1"
		\ar@{} (0,0);(35, -16)  *++! U{x_{2}} *\cir<4pt>{} = "X2"
		\ar@{} (0,0);(65, -16)  *++! U{x_{2k}} *\cir<4pt>{} = "X2b"
		\ar@{-} "X1";(20, -4) *++! D{y_{1}} *\cir<4pt>{} = "Y1";
		\ar@{-} "X2";(35, -4) *++!D{y_{2}} *\cir<4pt>{} = "Y2";
		\ar@{} "X1"; (49, 0) *++!U{\cdots}
		\ar@{} "X1"; (49, -12) *++!U{\cdots}
		\ar@{-} "X2b";(65, -4) *++!D{y_{2k}} *\cir<4pt>{} = "Y2b";
		\ar@{} (0,0);(80, -16)  *++! U{x_{2k + 1}} *\cir<4pt>{} = "X2b+1";
		\ar@{} (0,0);(110, -16)  *++! U{x_{2q - 1}} *\cir<4pt>{} = "X2a-1";
		\ar@{} (0,0);(125, -16)  *++! R{x_{2q}} *\cir<4pt>{} = "X2a";
		\ar@{-} "X2a";(140, -4) *++! L{z_{1}} *\cir<4pt>{} = "Z1";
		\ar@{-} "X2a";(140, -10) *++!L{z_{2}} *\cir<4pt>{} = "Z2";
		\ar@{-} "X2a";(140, -28) *++!U{z_{n - 2(q + k)}} *\cir<4pt>{} =  "Zc";
		\ar@{} "X1"; (140, -23) *++!D{\vdots}
		\ar@{} "X1"; (95, -12) *++!U{\cdots}
		\ar@{-} (15,-17.5); (126, -17.5);
		\ar@{-} (15,-17.5); (15, -35);
		\ar@{-} (15,-35); (126, -35);
		\ar@{-} (126,-17.5); (126, -35)
		\ar@{} (0,0);(74,-27.5) *{\Large{\text{$K_{2q}$}}};
	\end{xy}
  	\caption{The graph $G_{q, k, n - 2(q + k)}^{(1)}$}
  	\label{fig:G5}
	\end{figure} 
	\ \\ 	
	By virtue of Lemma \ref{G1}, one has 
	\begin{itemize}
		\item $\left| V\left( G_{q, k, n - 2(q + k)}^{(1)} \right) \right| = 2q + 2k + n - 2(q + k) = n$. 
		\item $\text{ind-match}\left( G_{q, k, n - 2(q + k)}^{(1)} \right) = 1$. 
		\item $\text{min-match}\left( G_{q, k, n - 2(q + k)}^{(1)} \right) = q$. 
		\item $\text{match}\left( G_{q, k, n - 2(q + k)}^{(1)} \right) = q + k = r$. 
	\end{itemize}
	Hence we have $\displaystyle (1, q, r) \in {\rm{Graph}_{ind\text{-}match, min\text{-}match, match}}(n)$. \\
	\item Assume that $p \geq 2$ and $p + q - r \leq 0$. Then note that $q - p + 1 > r - p - q \geq 0$, $n - 2r + 1 \geq 1$, $p - 1 \geq 1$.  
	Now we consider the graph $\displaystyle G_{q - p + 1, r - p - q, n - 2r + 1, p - 1, 0}^{(2)}$; see Figure \ref{fig:G6}. 
			\begin{figure}[htbp]
		\centering
		\begin{xy}
			\ar@{} (0,0);(15, -16)  *++! D{x_{1}} *\cir<4pt>{} = "X1"
			\ar@{} (0,0);(30, -16)  *++! RU{x_{2}} *\cir<4pt>{} = "X2"
			\ar@{} (0,0);(55, -16)  *++! RU{x_{2(r - p - q)}} *\cir<4pt>{} = "X2b"
			\ar@{-} "X1";(15, -28) *++! U{y_{1}} *\cir<4pt>{} = "Y1";
			\ar@{-} "X2";(30, -28) *++! U{y_{2}} *\cir<4pt>{} = "Y2";
			\ar@{} "X1"; (43, -24.5) *++!U{\cdots}
			\ar@{} "X1"; (43, -12.5) *++!U{\cdots}
			\ar@{-} "X2b";(55, -28) *++!U{y_{2(r - p - q)}} *\cir<4pt>{} = "Y2b";
			\ar@{} (0,0);(70, -16)  *++! U{x_{2(r - p - q) + 1}} *\cir<4pt>{} = "X2b+1";
			\ar@{} (0,0);(95, -16)  *++! U{x_{2(q - p)}} *\cir<4pt>{} = "X2a-2";
			\ar@{} (0,0);(110, -16)  *++! U{x_{2(q - p) + 1}} *\cir<4pt>{} = "X2a-1";
			\ar@{} (0,0);(125, -16) *\cir<4pt>{} = "X2a";
			\ar@{} (0,0);(123, -10) *{\text{$x_{2(q - p + 1)}$}};
			\ar@{-} "X2a";(140, -4) *++! L{z_{1}} *\cir<4pt>{} = "Z1";
			\ar@{-} "X2a";(140, -10) *++!L{z_{2}} *\cir<4pt>{} = "Z2";
			\ar@{-} "X2a";(140, -28) *++!U{z_{n - 2r + 1}} *\cir<4pt>{} =  "Zc";
			\ar@{} "X1"; (140, -23) *++!D{\vdots}
			\ar@{} "X1"; (82, -12.5) *++!U{\cdots}
			\ar@{-} (10,-17.5); (126, -17.5);
			\ar@{-} (10,-17.5); (10, -35);
			\ar@{-} (10,-35); (126, -35);
			\ar@{-} (126,-17.5); (126, -35)
			\ar@{} (0,0);(78,-28.5) *{\Large{\text{$K_{2(q - p +1)}$}}};
			\ar@{-} "X1";(70, 0) *++!D{w} *\cir<4pt>{} = "W";
			\ar@{-} "X2";"W";
			\ar@{-} "X2b";"W";
			\ar@{-} "X2b+1";"W";
			\ar@{-} "X2a-2";"W";
			\ar@{-} "X2a-1";"W";
			\ar@{-} "X2a";"W";
			\ar@{-} "W";(40, 16)  *++! R{u_{1}} *\cir<4pt>{} = "U1";
			\ar@{-} "W";(50, 16)  *++! L{u_{p}} *\cir<4pt>{} = "Ud+1";
			\ar@{-} "W";(90, 16)  *++! R{u_{p - 1}} *\cir<4pt>{} = "Ud";
			\ar@{-} "W";(100, 16)  *++! L{u_{2(p - 1)}} *\cir<4pt>{} = "U2d";
			\ar@{-} "U1";"Ud+1";
			\ar@{-} "Ud";"U2d";
			\ar@{-} "U1";(40, 26)  *++! D{u'_{1}} *\cir<4pt>{} = "U'1";
			\ar@{-} "Ud+1";(50, 26)  *++! D{u'_{p}} *\cir<4pt>{} = "U'd+1";
			\ar@{-} "Ud";(90, 26)  *++! D{u'_{p - 1}} *\cir<4pt>{} = "U'd";
			\ar@{-} "U2d";(100, 26)  *++! L{u'_{2(p - 1)}} *\cir<4pt>{} = "U'2d";
			\ar@{} "X1"; (70, 18) *++!D{\cdots}
		\end{xy}

 		 \caption{The graph $G_{q - p + 1, r - p - q, n - 2r + 1, p - 1, 0}^{(2)}$}
		  \label{fig:G6}
		\end{figure}
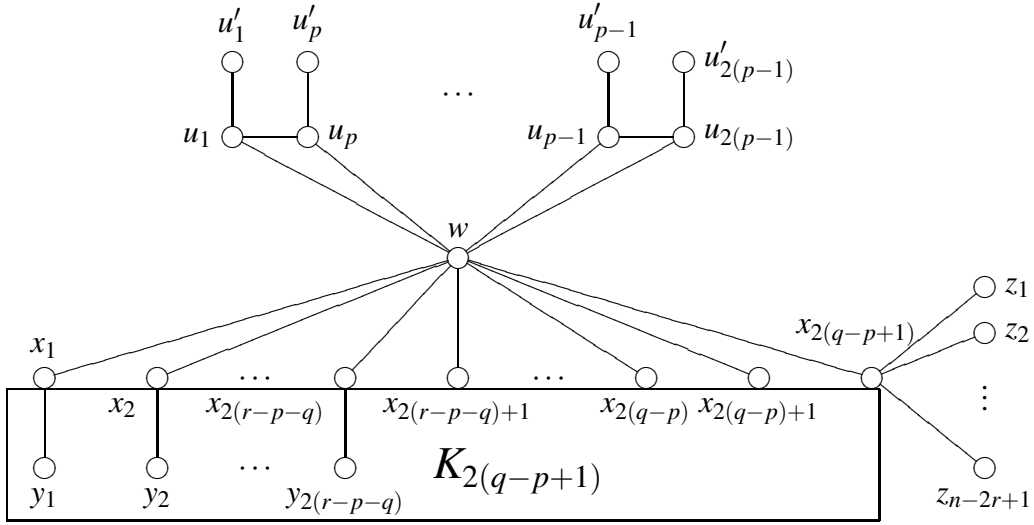
	By virtue of Lemma \ref{G2}, one has
		\begin{itemize}
			\item $\left| V\left(G_{q - p + 1, r - p - q, n - 2r + 1, p - 1, 0}^{(2)}\right) \right|$ \\ 
			$=  2(q - p + 1) + 2(r - p - q) + (n - 2r + 1) + 4(p - 1) + 1 = n$. 
			\item $\text{ind-match}\left(G_{q - p + 1, r - p - q, n - 2r + 1, p - 1, 0}^{(2)}\right)$  \\ 
			$= (p - 1) + 1 = p$.
			\item $\text{min-match}\left(G_{q - p + 1, r - p - q, n - 2r + 1, p - 1, 0}^{(2)}\right)$ \\
			$= (q - p + 1) + (p - 1) = q$. 
			\item $\text{match}\left(G_{q - p + 1, r - p - q, n - 2r + 1, p - 1, 0}^{(2)}\right)$ \\
			$= (q - p + 1) + (r - p - q) + 2(p - 1) + 1 = r$. 
		\end{itemize}
	Thus we have $\displaystyle (p, q, r) \in {\rm{Graph}_{ind\text{-}match, min\text{-}match, match}}(n)$. \\ 
 	\item Assume that $p \geq 2$, $p + q - r > 0$ and $q < r$. Then note that $q - p + 1 > 0$, $n - 2r + 1 \geq 1$, $r - q - 1 \geq 0$ and $p + q - r > 0$. 
	Now we consider the graph $\displaystyle G_{q - p + 1, 0, n - 2r + 1, r - q - 1, p + q - r}^{(2)}$; see Figure \ref{fig:G7}. 
			\begin{figure}[htbp]
		\centering
		\begin{xy}
			\ar@{} (0,0);(30, -16)  *++! U{x_{1}} *\cir<4pt>{} = "X1"
			\ar@{} (0,0);(55, -16)  *++! U{x_{2}} *\cir<4pt>{} = "X2"
			\ar@{} (0,0);(110, -16)   *\cir<4pt>{} = "X2a";
			\ar@{} (0,0);(108,-10) *{\text{$x_{2(q - p + 1)}$}};
			\ar@{-} "X2a";(125, -4) *++! L{z_{1}} *\cir<4pt>{} = "Z1";
			\ar@{-} "X2a";(125, -10) *++!L{z_{2}} *\cir<4pt>{} = "Z2";
			\ar@{-} "X2a";(125, -28) *++!L{z_{n - 2r + 1}} *\cir<4pt>{} =  "Zc";
			\ar@{} "X1"; (125, -23) *++!D{\vdots}
			\ar@{} "X1"; (82, -12.5) *++!U{\cdots}
			\ar@{-} (25,-17.5); (111, -17.5);
			\ar@{-} (25,-17.5); (25, -35);
			\ar@{-} (25,-35); (111, -35);
			\ar@{-} (111, -17.5); (111, -35)
			\ar@{} (0,0);(72,-27.5) *{\Large{\text{$K_{2(q - p + 1)}$}}};
			\ar@{-} "X1";(70, 0) *\cir<4pt>{} = "W";
			\ar@{} (0,0);(76,0) *{\text{$w$}};
			\ar@{-} "X2";"W";
			\ar@{-} "X2a";"W";
			\ar@{-} "W";(15, 16)  *++! R{u_{1}} *\cir<4pt>{} = "U1";
			\ar@{-} "W";(25, 16)  *++! L{u_{r - q}} *\cir<4pt>{} = "Ud+1";
			\ar@{-} "W";(50, 16)  *\cir<4pt>{} = "Ud";
			\ar@{} (0,0);(45, 12) *{\text{$u_{r - q - 1}$}};
			\ar@{-} "W";(60, 16)  *++! L{u_{2(r - q - 1)}} *\cir<4pt>{} = "U2d";
			\ar@{-} "W";(80, 16)  *++! D{v_{1}} *\cir<4pt>{} = "V1";
			\ar@{-} "W";(90, 16)  *\cir<4pt>{} = "Ve+1";
			\ar@{} (0,0);(92, 20) *{\text{$v_{p + q - r + 1}$}};
			\ar@{-} "W";(115, 16)  *++! D{v_{p + q - r}} *\cir<4pt>{} = "Ve";
			\ar@{-} "W";(125, 16)  *++! U{v_{2(p + q - r)}} *\cir<4pt>{} = "V2e";
			\ar@{-} "U1";"Ud+1";
			\ar@{-} "Ud";"U2d";
			\ar@{-} "V1";"Ve+1";
			\ar@{-} "Ve";"V2e";
			\ar@{-} "U1";(15, 26)  *++! D{u'_{1}} *\cir<4pt>{} = "U'1";
			\ar@{-} "Ud+1";(25, 26)  *++! D{u'_{r - q}} *\cir<4pt>{} = "U'd+1";
			\ar@{-} "Ud";(50, 26)  *++! D{u'_{r - q - 1}} *\cir<4pt>{} = "U'd";
			\ar@{-} "U2d";(60, 26)  *\cir<4pt>{} = "U'2d";
			\ar@{} (0,0);(68, 30) *{\text{$u'_{2(r - q - 1)}$}};
			\ar@{} "X1"; (38, 18) *++!D{\cdots}
			\ar@{} "X1"; (103, 12) *++!D{\cdots}
		\end{xy}

		  \caption{The graph $G_{q - p + 1, 0, n - 2r + 1, r - q - 1, p + q - r}^{(2)}$}
		  \label{fig:G7}
		\end{figure}
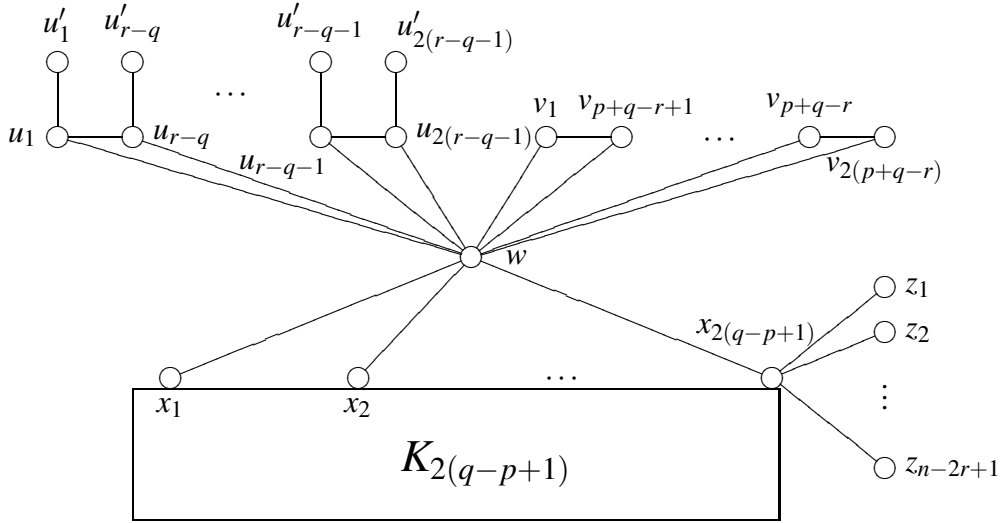
	By virtue of Lemma \ref{G2}, one has 
		\begin{itemize}
			\item $\left| V\left(G_{q - p + 1, 0, n - 2r + 1, r - q - 1, p + q - r}^{(2)}\right) \right|$ \\ 
			$=  2(q - p + 1) + (n - 2r + 1) + 4(r - q - 1) + 2(p + q - r) + 1 = n$. 
			\item $\text{ind-match}\left(G_{q - p + 1, 0, n - 2r + 1, r - q - 1, p + q - r}^{(2)}\right)$  \\ 
			$= (r - q - 1) + (p + q - r) + 1 = p$.
			\item $\text{min-match}\left(G_{q - p + 1, 0, n - 2r + 1, r - q - 1, p + q - r}^{(2)}\right)$ \\
			$= (q - p + 1) + (r - q - 1) + (p + q - r) = q$. 
			\item $\text{match}\left(G_{q - p + 1, 0, n - 2r + 1, r - q - 1, p + q - r}^{(2)}\right)$ \\
			$= (q - p + 1) + 2(r - q - 1) + (p + q - r) + 1 = r$. 
		\end{itemize}
	Thus we have $\displaystyle (p, q, r) \in {\rm{Graph}_{ind\text{-}match, min\text{-}match, match}}(n)$. 
	\item Assume that $p \geq 2$ and $q = r$. Note that $q - p + 1 \geq 1$, $p - 2 \geq 0$ and $n - 2q \geq 1$. Now we consider the graph $\displaystyle G_{q - p + 1, p - 2, n - 2q}^{(3)}$; see Figure \ref{fig:G8}:
			\ \\ 
		\begin{figure}[htbp]
		\centering
		\begin{xy}
			\ar@{} (0,0);(30, -16)  *++! U{x_{1}} *\cir<4pt>{} = "X1"
			\ar@{} (0,0);(55, -16)  *++! U{x_{2}} *\cir<4pt>{} = "X2"
			\ar@{} (0,0);(110, -16)  *++! D{x_{2(q - p + 1)}} *\cir<4pt>{} = "X2a";
			\ar@{} "X1"; (82, -12.5) *++!U{\cdots}
			\ar@{-} (25,-17.5); (115, -17.5);
			\ar@{-} (25,-17.5); (25, -35);
			\ar@{-} (25,-35); (115, -35);
			\ar@{-} (115,-17.5); (115, -35)
			\ar@{} (0,0);(70,-27.5) *{\Large{\text{$K_{2(q - p + 1)}$}}};
			\ar@{-} "X1";(70, 0) *++! R{w} *\cir<4pt>{} = "W";
			\ar@{-} "X2";"W";
			\ar@{-} "X2a";"W";
			\ar@{-} "W";(45, 16)  *++! U{v} *\cir<4pt>{} = "V";
			\ar@{-} "V";(30, 24) *++! D{z_{1}} *\cir<4pt>{} = "Z1";
			\ar@{-} "V";(40, 24) *++!D{z_{2}} *\cir<4pt>{} = "Z2";
			\ar@{-} "V";(60, 24) *++!D{z_{n - 2q}} *\cir<4pt>{} =  "Zc";
			\ar@{-} "W";(80, 16)  *++! D{y_{1}} *\cir<4pt>{} = "Y1";
			\ar@{-} "W";(90, 16)  *++! D{y_{p - 1}} *\cir<4pt>{} = "Yb+1";
			\ar@{-} "W";(115, 16)  *++! D{y_{p - 2}} *\cir<4pt>{} = "Yb";
			\ar@{-} "W";(125, 16)  *++! U{y_{2(p - 2)}} *\cir<4pt>{} = "Y2b";
			\ar@{-} "Y1";"Yb+1";
			\ar@{-} "Yb";"Y2b";
			\ar@{} "X1"; (50, 20.5) *++!D{\cdots}
			\ar@{} "X1"; (103, 12) *++!D{\cdots}
		\end{xy}

		  \caption{The graph $G_{q - p + 1, p - 2, n - 2q}^{(3)}$}
		  \label{fig:G8}
		\end{figure}
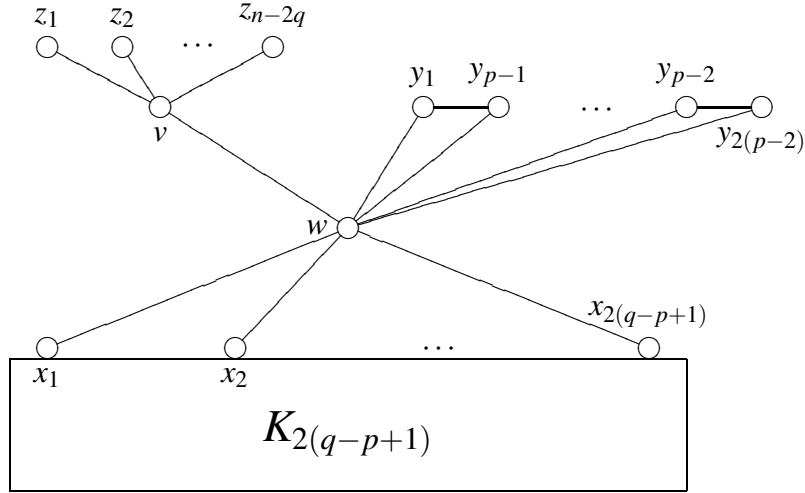
		\ \\ 
	By virtue of Lemma \ref{G3}, one has 
		\begin{itemize}
			\item $\left| V\left(G_{q - p + 1, p - 2, n - 2q}^{(3)}\right) \right| = 2(q - p + 1) + 2(p - 2) + (n - 2q) + 2 = n$. 
			\item $\text{ind-match}\left(G_{q - p + 1, p - 2, n - 2q}^{(3)}\right) = (p - 2) + 2 = p$. 
			\item $\text{min-match}\left(G_{q - p + 1, p - 2, n - 2q}^{(3)}\right) = \text{match}\left(G_{a, b, c}^{(3)}\right)$ \\  $= (q - p + 1) + (p - 2) + 1 = q$. 
		\end{itemize}
	Thus we have $\displaystyle (p, q, r) \in {\rm{Graph}_{ind\text{-}match, min\text{-}match, match}}(n)$.
\end{itemize}

Next, we assume that $n$ is even and 
\begin{eqnarray*}
(p, q, r) &\in& \left\{ (1, q, r) \in \mathbb{N}^{3} \ \middle| \ 1 \leq q \leq r \leq 2q \ \ {\rm{and}} \ \ r \leq \frac{n}{2} \right\} \\
& & \cup \left\{ (p, q, r) \in \mathbb{N}^{3} \ \middle| \ 2 \leq p \leq q \leq r \leq 2q, \ r \leq \frac{n}{2}\ \ {\rm{and}}\ \ (q, r) \neq \left( \frac{n}{2}, \frac{n}{2} \right) \right\} . 
\end{eqnarray*}
As in the case that $n$ is odd, we can see that $\displaystyle (p, q, r) \in {\rm{Graph}_{ind\text{-}match, min\text{-}match, match}}(n)$ by considering $G_{q, k, n - 2(q + k)}^{(1)}$, $G_{q - p + 1, r - p - q, n - 2r + 1, p - 1, 0}^{(2)}$, $G_{q - p + 1, 0, n - 2r + 1, r - q - 1, p + q - r}^{(2)}$ and $G_{q - p + 1, p - 2, n - 2q}^{(3)}$. \\

Therefore we have the desired conclusion. 
\end{proof}

\section{The set ${\rm{\bf Graph}_{reg, min\text{-}match, match}}(n)$}

In this section, as an application of Theorem \ref{first-main}, we determine the possible tuples 
\[
(\reg(G), \text{min-match}(G), \text{match}(G), |V(G)|)
\]
arising from connected simple graphs. 

Let $G$ be a finite simple graph on the vertex set $V(G) = \left\{ x_{1}, \ldots, x_{|V(G)|} \right\}$ and $E(G)$ the set of edges of $G$. 
Let $K[V(G)] = K\left[ x_{1}, \ldots, x_{|V(G)|} \right]$ be the polynomial ring in $|V(G)|$ variables over a field $K$. 
Now we associate with $G$ the quadratic monomial ideal 
\[
I(G) = \left( x_{i}x_{j} \mid \{x_{i}, x_{j}\} \in E(G) \right) \subset K[V(G)]. 
\]
The ideal $I(G)$ is called the {\em edge ideal} of $G$. 


The relationship between graph-theoretical invariant of $G$ and ring-theoretical invariants of the quotient ring $K[V(G)]/I(G)$ has been studied.  
As previous results,  
\begin{itemize}
	\item In \cite[Theorem 1]{HM}, Hirano and the first-named author determined the possible tuples 
	\[
	(\text{ind-match}(G), \text{min-match}(G), \text{match}(G), \dim(G))
	\]
	arising from connected simple graphs, where $\dim(G) = \dim K[V(G)]/I(G)$ denote the Krull dimension of $K[V(G)]/I(G)$. 
	\item In \cite{HMVT}, Hibi et al. proved that 
	\[
	\deg(G) + \reg(G) \leq |V(G)|
	\]
	for all simple graph $G$, 
	 where $\deg(G) = \deg h_{K[V(G)]/I(G)}(t)$ denote the degree of the $h$-polynomial of $K[V(G)]/I(G)$.  
	\item In \cite{HKMVT}, Hibi et al. studied the possible tuples $(\reg(G), \deg(G), |V(G)|)$ arising from connected simple graphs $G$ and determined this tuples arising from {\em Cameron--Walker} graphs, where a finite connected simple graph $G$ is said to be a Cameron--Walker graph if $\text{ind-match}(G) = \text{min-match}(G) = \text{match}(G)$ and if $G$ is neither a star graph nor a star triangle. 
		\item In \cite{HKKMVT}, Hibi et al. studied the possible tuples $(\depth(G), \dim(G), |V(G)|)$ arising from connected simple graphs $G$ and determined this tuples arising from Cameron--Walker graphs. They also determined the possible tuples 
	\[
	(\depth(G), \reg(G), \dim G, \deg(G), |V(G)|)
	\]
	arising from Cameron--Walker graphs, where $\depth(G) = \depth(K[V(G)]/I(G))$ denote the depth of $K[V(G)]/I(G)$; 
	\item Erey--Hibi determined the possible tuples $(\text{pd}(G), \reg(G), |V(G)|)$ arising from connected {\em bipartite} graphs, where $\text{pd}(G) = \text{pd}(K[V(G)]/I(G))$ denote the projective dimension of $K[V(G)]/I(G)$ (\cite[Theorem 3.14]{EH}). This tuples also studied in \cite{HaHi}. \\
\end{itemize}


The second main result is as follows. 
We determine the possible tuples 
\[
(\reg(G), \text{min-match}(G), \text{match}(G), |V(G)|)
\]
arising from connected simple graphs. 

\begin{Theorem}\label{second-main}
Let $n \geq 2$ be an integer and set 
\begin{eqnarray*}
& & {\rm{\bf Graph}_{reg, min\text{-}match, match}}(n) \\
&=& \left\{(p', q, r) \in \mathbb{N}^{3} ~\left|~
\begin{array}{c}
  \mbox{\rm{There\ exists\ a\ connected\ simple\ graph}\ $G$\ {\rm{with}}\ $|V(G)| = n$ } \\
  \mbox{{\rm{and}} \ $\text{\rm reg}(G) = p', \ \text{\rm min-match}(G) = q, \ \text{\rm match}(G) = r$} \\ 
\end{array}
\right \}\right. .  
\end{eqnarray*}
Then one has 
\[
{\rm{\bf Graph}_{reg, min\text{-}match, match}}(n) = {\rm{\bf Graph}_{ind\text{-}match, min\text{-}match, match}}(n). 
\]
\end{Theorem}
\begin{proof}
From \cite{W} and Proposition \ref{important}(1), we have that 
\[
\reg(G) \leq \text{min-match}(G) \leq \text{match}(G) \leq 2\text{min-match}(G)
\]
holds for all connected graph $G$. 
Since the graphs $G_{q, k, n - 2(q + k)}^{(1)}$, $G_{q - p + 1, r - p - q, n - 2r + 1, p - 1, 0}^{(2)}$, $G_{q - p + 1, 0, n - 2r + 1, r - q - 1, p + q - r}^{(2)}$ and $G_{q - p + 1, p - 2, n - 2q}^{(3)}$ which appeared in the proof of Theorem \ref{first-main} are chordal, hence it follows that 
the regularity of these graphs equal to its induced matching number. 
Moreover, since both of the complements of $K_{n}$ and $K_{n/2, n/2}$ are chordal, 
one has $\reg(K_{n}) = \reg(K_{n/2, n/2}) = 1$ by virtue of Fr\"oberg \cite{F}. 
Thus, by using Proposition \ref{min = n/2}, we have that there is no connected simple graph $G$ with
\[ 
\left( \reg(G), \text{min-match}(G), \text{match}(G), |V(G)| \right) = (p, n/2, n/2, n)
\]
 for all $p \geq 2$. 
Therefore we have the desired conclusion. 
\end{proof}

\bigskip

\noindent
{\bf Acknowledgment.}
Kazunori Matsuda was partially supported by JSPS Grants-in-Aid for Scientific Research (JP20K03550, JP20KK0059).


\bigskip


\begin{thebibliography}{99}

\bibitem{AA}
T.~C.~Adefokun and D.~O.~Ajavi, 
On maximum induced matching numbers of special grids,  
{\em J. Math. Appl.} {\bf 41} (2018), 5--18.  

\bibitem{AV}
S.~Arumugam and S.~Velammal, 
Edge domination in graphs,  
{\em Taiwanese J. Math.} {\bf 2} (1998), 173--179.  

\bibitem{BCL}
R.~Boliac, K.~Cameron and V.~V.~Lozin, 
On computing the dissociation number and the induced matching number of bipartite graphs,  
{\em Ars Combin.} {\bf 72} (2004), 241--253.  

\bibitem{CGH}
S.~M.~Cioab\u{a}, D.~A.~Gregory and W.~H.~Haemers, 
Matchings in regular graphs from eigenvalues, 
{\em J. Combin. Theory Ser. B} {\bf 99} (2009), 287--297.  

\bibitem{CRT}
M. Crupi, G. Rinaldo and N. Terai, 
Cohen--Macaulay edge ideal whose height is half of the number of vertices, 
{\em Nagoya Math. J.} {\bf 201} (2011), 117--131. 

\bibitem{CRTY}
M.~Crupi, G.~Rinaldo, N.~Terai and K.~Yoshida, 
Effective Cowsik-Nori theorem for edge ideals, 
{\em Comm. Algebra} {\bf 38} (2010), 3347--3357. 

\bibitem{CW}
K.~Cameron and T.~Walker, 
The graphs with maximum induced matching and maximum matching the same size, 
{\em Discrete Math.} {\bf 299} (2005), 49--55. 

\bibitem{DHS}
H. Dao, C. Huneke and J. Schweig, 
Bounds on the regularity and projective dimension of ideals associated to graphs, 
{\em J. Algebraic Combin.} {\bf 38} (2013), 37--55. 

\bibitem{DK}
R.~Dutton and W.~F.~Klostermeyer, 
Edge dominating sets and vertex covers, 
{\em Discuss. Math. Graph Theory} {\bf 33} (2013), 437--456.  

\bibitem{EH}
N. Erey and T. Hibi, 
The size of Betti tables of edge ideals arising from bipartite graphs, 
arXiv:2103.04766.  

\bibitem{F}
R. Fr\"{o}berg, 
{\it On Stanley-Reisner rings.} 
Topics in algebra, Banach Center Publications, {\bf 26} (1990), 57--70. 

\bibitem{FR}
M.~F\"{u}rst and D.~Rautenbach, 
On the equality of the induced matching number and the uniquely restricted matching number for subcubic graphs, 
{\em Theoret. Comput. Sci.} {\bf 804} (2020), 126--138. 

\bibitem{GR}
C.~Godsil and G.~Royle,  
{\em Algebraic graph theory. }
Graduate Texts in Mathematics {\bf 207}. 
Springer-Verlag, New York, 2001. 

\bibitem{HaHi}
H.~T.~H\`{a} and T.~Hibi, 
MAX MIN Vertex Cover and the Size of Betti Tables, 
{\em Ann. Comb. } {\bf 25} (2021), 115--132. 

\bibitem{HaVanTuyl}
H. T. H\`{a} and A. Van Tuyl, 
Monomial ideals, edge ideals of hypergraphs, and their graded Betti numbers, 
{\em J. Algebraic Combin.} {\bf 27} (2008), 215--245. 

\bibitem{Hall}
P.~Hall,  
On representatives of subsets,  
{\em J. Lond. Math. Soc.} {\bf 10(1)} (1935), 26--30.  

\bibitem{HHKO}
T.~Hibi, A.~Higashitani, K.~Kimura and A.~B.~O'Keefe,
Algebraic study on Cameron--Walker graphs, 
{\em J. Algebra} {\bf 422} (2015), 257--269. 

\bibitem{HHKT}
T.~Hibi, A.~Higashitani, K.~Kimura and A.~Tsuchiya, 
Dominating induced matchings of finite graphs and regularity of edge ideals, 
{\em J. Algebraic Combin.} {\bf 43} (2016), 173--198. 

\bibitem{HKKMVT}
T. Hibi, H. Kanno, K. Kimura, K. Matsuda and A. Van Tuyl, 
Homological invariants of Cameron--Walker graphs, 
{\em Trans. Amer. Math. Soc.} {\bf 374} (2021), no. 9, 6559--6582. 

\bibitem{HKM}
T.~Hibi, H.~Kanno and K.~Matsuda, 
Induced matching number of finite graphs and edge ideals, 
{\em J. Algebra} {\bf 532} (2019), 311--322. 

\bibitem{HKMT}
T.~Hibi, K.~Kimura, K.~Matsuda and A.~Tsuchiya, 
Regularity and $a$-invariant of Cameron--Walker graphs, 
{\em J. Algebra} {\bf 584} (2021), 215--242. 

\bibitem{HKMVT}
T.~Hibi, K.~Kimura, K.~Matsuda and A.~Van Tuyl, 
The regularity and $h$-polynomial of Cameron--Walker graphs, 
arXiv:2003.07416. 

\bibitem{HMVT}
T.~Hibi, K.~Matsuda and A.~Van Tuyl, 
Regularity and $h$-polynomials of edge ideals, 
{\em Electron. J. Combin.} {\bf 26} (2019), Paper 1.22, 11 pages. 

\bibitem{HM}
A. Hirano and K. Matsuda, 
Matching numbers and dimension of edge ideals, 
{\em Graphs Combin. } {\bf 37} (2021), 761--774. 

\bibitem{K}
M.~Katzman, 
Characteristic-independence of Betti numbers of graph ideals, 
{\em J. Combin. Theory Ser. A} {\bf 113} (2006), 435--454. 

\bibitem{KM}
F. Khosh-Ahang and S. Moradi, 
Regularity and projective dimension of the edge ideal of $C_5$-free vertex decomposable graphs, 
{\em Proc. Amer. Math. Soc.} {\bf 142} (2014), no. 5, 1567--1576. 

\bibitem{LV}
M.~Las Vergnas, 
A note on matchings in graphs, 
Colloque sur la Th\'{e}orie des Graphes (Paris 1974), 
{\em Cahiers Centre \'{E}tudes Rech. Op\'{e}r. } {\bf 17} (1975), 257--260.  

\bibitem{MV}
S.~Morey and R.~H.~Villarreal, 
Edge ideals: algebraic and combinatorial properties, in: Progress in commutative algebra 1, 
de Gruyter, Berlin, 2012, pp.85--126. 

\bibitem{NeP}
E.~Nevo and I.~Peeva, 
$C_{4}$-free edge ideals, 
{\em J. Algebraic Combin.} {\bf 37} (2013), 243--248. 

\bibitem{P}
I.~Peeva, 
{\em Graded syzygies. }
Algebra and Applications {\bf 14}. 
Springer, London, 2011. 

\bibitem{Romeo}
F.~Romeo, 
Chordal circulant graphs and induced matching number, 
{\em Discrete Math.} {\bf 343} (2020), 111947, 6 pp. 

\bibitem{SF}
S.~A.~Seyed Fakhari, 
Regularity of symbolic powers of edge ideals of Cameron--Walker graphs, 
{\em Comm. Algebra} {\bf 48} (2020), 5215--5223. 

\bibitem{SVV}
A.~Simis, W.~V.~Vasconcelos and R.~H.~Villarreal, 
On the ideal theory of graphs, 
{\em J. Algebra} {\bf 167} (1994), 389--416. 

\bibitem{SMWZ}
W.~Song, L.~Miao, H.~Wang and Y.~Zhao, 
Maximal matching and edge domination in complete multipartite graphs, 
{\em Int. J. Comput. Math.} {\bf 91} (2014), 857--862.  

\bibitem{Sumner}
D.~P.~Sumner, 
Graphs with $1$-factors, 
{\em Proc. Amer. Math. Soc.} {\bf 42} (1974), 8--12.  

\bibitem{T}
T.~N.~Trung, 
Regularity, matchings and Cameron--Walker graphs. 
{\em Collect. Math.} {\bf 71} (2020), 83--91. 

\bibitem{Vi2001}
R.~H.~Villarreal, 
Monomial algebras, 
Monographs and Textbooks in Pure and applied Mathematics {\bf 238}, 
Marcel Dekker, Inc., New York, 2001.

\bibitem{W}
R. Woodroofe, 
Matchings, coverings, and Castelnuovo-Mumford regularity, 
{\em J. Commut. Algebra} {\bf 6} (2014), 287--304. 


\end{thebibliography}
\end{document}